\documentclass[english,a4paper,12pt]{amsart}
\usepackage[a4paper,lmargin=2.5cm,rmargin=2.5cm,tmargin=4cm,bmargin=4cm]{geometry}
\usepackage{amsfonts}
\usepackage{amssymb}
\usepackage[colorlinks]{hyperref}
\usepackage{tikz}
\usepackage[normalem]{ulem}
\usepackage{tkz-euclide}
\usepackage{bbm}
\usepackage{mathabx}

\usepackage{xspace} % add space after e.g. \RNP as necessary
\usepackage{enumitem} % more enumerate/itemize options
\setenumerate[1]{label=(\roman{*}), ref=(\roman{*})}
\setenumerate[2]{label=(\alph{*}), ref=(\alph{*})}
\usepackage{marginnote}

\definecolor{egraf}{rgb}{0.2,0.4,0}
\newcommand{\abs}[1]{\left\vert#1\right\vert}
\newcommand{\set}[1]{\left\{#1\right\}}

\newcommand\restr[2]{{% we make the whole thing an ordinary symbol
    \left.\kern-\nulldelimiterspace % automatically resize the bar with \right
      #1 % the function
      \littletaller % pretend it's a little taller at normal size
    \right|_{#2} % this is the delimiter
  }}
\newcommand{\littletaller}{\mathchoice{\vphantom{\big|}}{}{}{}}
\newcommand{\diam}{\mathop{\mathrm{diam}}\nolimits}

\newcommand{\norm}[1]{\left\Vert#1\right\Vert}
\newcommand{\duality}[1]{\left\langle#1\right\rangle}
\newcommand{\csp}[1]{\overline{\mathrm{span}}^{#1}}
\newcommand{\linsp}{{\mathrm{span}}}

\newcommand{\Lip}{{\mathrm{Lip}}}
\newcommand{\supp}{\mathop\mathrm{supp}}
\newcommand{\ext}{\mathop\mathrm{ext}}

\newcommand{\eps}{\varepsilon}

\DeclareMathOperator{\conv}{conv}
\DeclareMathOperator{\clco}{\overline{conv}}

\newcommand{\lipfree}[1]{\mathcal{F}({#1})}                 % Lipschitz-free space
\newcommand{\lipnorm}[1]{\norm{#1}_L}                       % Lipschitz norm/constant
\newcommand{\cl}[1]{\overline{#1}}                          % closure
\DeclareMathOperator{\lip}{lip}                             % little
\newcommand{\N}{\mathbb N}
\newcommand{\R}{\mathbb R}

\newcommand{\RNP}{Radon--Nikod\'{y}m property\xspace}

\newcommand{\nnorm}[1]{\left\vvvert#1\right\vvvert}

\theoremstyle{plain}
\newtheorem{thm}{Theorem}[section]
\newtheorem{cor}[thm]{Corollary}
\newtheorem{lem}[thm]{Lemma}
\newtheorem{prop}[thm]{Proposition}

\newtheorem*{thm*}{Theorem}
\newtheorem*{prop*}{Proposition}

\theoremstyle{definition}
\newtheorem{defn}[thm]{Definition}
\newtheorem{rem}[thm]{Remark}

\newtheorem{example}[thm]{Example}

\newtheorem{quest}[thm]{Question}

\author[T.~A.~Abrahamsen]{Trond A.~Abrahamsen}
\address[T.~A.~Abrahamsen]{Department of Mathematics, University of
  Agder, Postboks 422, 4604 Kristiansand, Norway.}
\email{trond.a.abrahamsen@uia.no}
\urladdr{http://home.uia.no/trondaa/index.php3}

\author[R.~J.~Aliaga]{Ram\'on J. Aliaga}
\address[R.~J.~Aliaga]{Instituto Universitario de Matem\'atica Pura y Aplicada,
Universitat Polit\`ecnica de Val\`encia,
Camino de Vera S/N,
46022 Valencia, Spain}
\email{raalva@upv.es}

\author[V.~Lima]{Vegard Lima}
\address[V.~Lima]{Department of Engineering Sciences, University of Agder,
Postboks 509, 4898 Grimstad, Norway.}
\email{vegard.lima@uia.no}

\author[A.~Martiny]{Andr\'e Martiny}
\address[A.~Martiny]{Department of Mathematics, University of
  Agder, Postboks 422, 4604 Kristiansand, Norway.}
\email{andre.martiny@uia.no}

\author[Y.~Perreau]{Yoël Perreau}
\address[Y.~Perreau]{University of Tartu, Institute of Mathematics and Statistics, Narva mnt 18, 51009 Tartu
  linn, Estonia}
\email{yoel.perreau@ut.ee}

\author[A.~Prochazka]{Anton\'in Prochazka}
\address[A.~Prochazka]{Université de Franche-Comté,
Laboratoire de mathématiques de Besançon,
UMR CNRS 6623,
16 route de Gray,
25000 Besançon,
France}
\email{antonin.prochazka@univ-fcomte.fr}

\author[T.~Veeorg]{Triinu Veeorg}
\address[T.~Veeorg]{University of Tartu, Institute of Mathematics and Statistics, Narva mnt 18, 51009 Tartu
  linn, Estonia}
\email{triinu.veeorg@ut.ee}

\title{Delta-points and their implications for the geometry of Banach spaces}
\begin{document}

\begin{abstract}
  We show that the Lipschitz-free space with the
  Radon--Nikod\'{y}m property
  and a Daugavet point recently constructed by Veeorg
  is in fact a dual space isomorphic to $\ell_1$.
  Furthermore, we answer an open problem from the literature
  by showing that there exists a superreflexive
  space, in the form of a renorming of $\ell_2$, with a
  $\Delta$-point.
  Building on these two results, we are able to renorm
  every infinite-dimensional Banach space to have a $\Delta$-point.

  Next, we establish powerful relations between existence of $\Delta$-points in
  Banach spaces and their duals. As an application, we obtain sharp results about
  the influence of $\Delta$-points for the asymptotic geometry of Banach spaces.
  In addition, we prove that if $X$ is a Banach space with a shrinking
  $k$-unconditional basis with $k < 2$, or if $X$ is a Hahn--Banach
  smooth space with a dual satisfying the Kadets--Klee property,
  then $X$ and its dual $X^*$ fail to contain $\Delta$-points. In particular, we get
  that no Lipschitz-free space with a Hahn--Banach smooth predual
  contains $\Delta$-points.
  
  Finally we present a purely metric characterization of the molecules
  in Lipschitz-free spaces that are $\Delta$-points, and we solve an
  open problem about representation of finitely supported
  $\Delta$-points in Lipschitz-free spaces.
\end{abstract}
\maketitle
\markleft{\textsc{Abrahamsen et al.}}

\section{Introduction}
\label{sec:introduction}
Let $X$ be a (real) Banach space with unit ball $B_X$, unit sphere $S_X$
and topological dual $X^*$.
For $x \in S_X$ we will write
\begin{equation*}
  D(x) := \set{x^* \in S_{X^*} : x^*(x) = 1}.
\end{equation*}
A \emph{slice} of a non-empty, bounded and convex subset $C$ of $X$
is a non-empty intersection of $C$ with an open half-space of $X$.
Thus a slice of $B_X$ can be written
\begin{equation*}
  S(B_X,x^*,\varepsilon) = \set{ y \in B_X : x^*(y) > 1 - \varepsilon}
\end{equation*}
where $x^* \in S_{X^*}$ and $\varepsilon > 0$.
We usually omit $B_X$ and write $S(x^*,\varepsilon)$ instead
of $S(B_X,x^*,\varepsilon)$ when it is clear from the
context what set is being sliced.
If $X$ is a dual space and the defining functional $x^*$
is in the predual of $X$, then we call the slice a weak$^*$-slice.

The main characters in our story are pointwise versions of
the well-known Daugavet property,
the slightly less known space with bad projections,
and the more obscure property $\mathfrak{D}$.
We start by recalling the definitions from \cite{AHLP}
and \cite{MPRZ}.

\begin{defn}\label{defn:slice_diametral_points}
  Let $X$ be a Banach space, and let $x\in S_X$.
  We say that
  \begin{enumerate}
  \item
    $x$ is a \emph{Daugavet point} if
    $\sup_{y \in S} \norm{x - y} = 2$ for every slice $S$ of $B_X$.
  \item
    $x$ is a \emph{$\Delta$-point} if
    $\sup_{y \in S}\norm{x - y} = 2$ for every slice $S$ of $B_X$
    with $x \in S$.
  \item $x$ is a \emph{$\mathfrak{D}$-point} if
    $\sup_{y \in S}\norm{x - y} = 2$ for every slice
    $S = S(x^*,\varepsilon)$ of $B_X$ with $x^* \in D(x)$
    and $\varepsilon > 0$.
  \item
    $x$ is a \emph{super $\Delta$-point}
    if $\sup_{y \in W}\norm{x - y} = 2$ for every
    relatively weakly open subset $W$ of $B_X$ with $x \in W$.
  \end{enumerate}
\end{defn}
For dual spaces we will also consider the natural weak$^*$ versions of
Daugavet and $\Delta$-points where we simply replace
the phrase ``every slice $S$'' in the definition with
``every weak$^*$-slice $S$'',
and replace ``weakly open'' by ``weak$^*$ open'' for super $\Delta$-points.

Note that if $X$ is a subspace of a Banach space $Y$
and $x \in S_X$ is a $\Delta$-point, $\mathfrak{D}$-point,
or super $\Delta$-point, then $x$ is still such a point
regarded as an element in $Y$.
This is not the case for Daugavet points as can be
seen by regarding $C[0,1]$ as a subspace of $C[0,1] \oplus_2 C[0,1]$
(this is Example~4.7 from \cite{AHLP}).

In Sections~\ref{sec:separable-dual-space}
through~\ref{sec:renorm-Banach-spaces} we show that Daugavet points,
super $\Delta$-points and $\Delta$-points can exist in some well
behaved Banach spaces.
The role played by $\mathfrak{D}$-points will be more negative,
in the sense that spaces with certain properties do not
even admit a $\mathfrak{D}$-point.

We start the paper with a study of the metric space $\mathcal{M}$
constructed by Veeorg in \cite{VeeorgStudia}.
The Lipschitz-free space $\lipfree{\mathcal{M}}$ was the first example of
a Banach space with the \RNP admitting a Daugavet point.
We will show that in fact $\lipfree{\mathcal{M}}$ is
isomorphic to $\ell_1$, and is isometrically a dual space.
Thus there exits a separable dual space with a Daugavet point
(see Theorem~\ref{thm:veeorg_space}).

It was shown in \cite[Corollary~6.10]{ALMP} that finite dimensional
Banach spaces do not admit $\Delta$-points, and it was asked
if the same holds for superreflexive spaces in \cite[Question~6.1]{ALMP}
and \cite[Question~7.7]{MPRZ}.
In Section~\ref{sec:superr-banach-space} we answer this question
negatively. We modify a renorming of $\ell_2$ from \cite{DKR+}
to show that $\ell_2$ can be renormed so that both $\ell_2$
and its dual have a super $\Delta$-point.
These super $\Delta$-points are not Daugavet points as there are
strongly exposed bits of the original unit ball which are still left
in the new unit ball and which are at distance strictly less than 2 to them. 
We also provide a positive answer to \cite[Question~7.12]{MPRZ} by proving that those super $\Delta$-points
actually belong to the closure of the set of strongly exposed-points for the new norm.

In Section~\ref{sec:renorm-Banach-spaces} we combine
the ideas and results from the previous two sections and show that
any infinite dimensional Banach space admits a renorming
with a $\Delta$-point.
Spaces failing the Schur property (and in particular
spaces which do not contain a copy of $\ell_1$) always contain
a normalized weakly null basic sequence, and using this as a starting point,
we can adapt the (dual) renorming of $\ell_2$ from Section~\ref{sec:superr-banach-space}
to get a renorming with a super $\Delta$-point.
For spaces containing $\ell_1$, the $\ell_1$ isomorphism from 
Section~\ref{sec:separable-dual-space}  together with a classic norm extension result
will allow for a renorming with a $\Delta$-point.

Having established that Daugavet points and $\Delta$-points
are very much isometric notions we go looking for
conditions that will prevent the existence of
such points in a Banach space.
Let us mention that it is known that neither
uniformly non-square \cite[Corollary~2.4]{ALMP},
asymptotically uniformly smooth Banach spaces \cite[Theorem~3.7]{ALMP}
nor real (GM) polyhedral spaces \cite[Corollary~3.8]{MR4405563}
contain $\Delta$-points.
It is obvious that no strongly exposed point can be a
$\mathfrak{D}$-point and that no denting point can be a
$\Delta$-point and, in fact, neither can
quasi-denting points by \cite[Corollary~2.2]{VeeorgFunc}. It was
also recently proved in \cite[Theorem~4.2]{KLT} that no locally 
uniformly non-square point can be a $\Delta$-point.

The main result of Section~\ref{sec:DualityForDelta} is
Theorem~\ref{thm:D-point/weak*-Delta-points_implies_weak*-super-Delta-point}
which says that as soon as a space contains a $\mathfrak{D}$-point
or its dual contains a weak$^*$ $\Delta$-point, then
the dual actually contains a weak$^*$ super $\Delta$-point.
This powerful result will have many implications
so let us mention a few.

In Subsection~\ref{subsec:Delta-duality_and_asymptotic-geometry} we prove
Theorem~\ref{thm:D-point/weak*-Delta-points_implies_weak*-super-Delta-point}
and use it to improve results from \cite{ALMP} and \cite{VeeorgFunc}
and show that asymptotically uniformly smooth spaces
cannot contain $\mathfrak{D}$-points and their duals
cannot contain weak$^*$ $\Delta$-points.

It is well-known that a separable Banach space with the Daugavet
property does not embed into a Banach space with an unconditional
basis \cite[Corollary~2.7]{MR1621757}.
However, for Daugavet points there is no such obstruction.
There exists a Banach space with a 1-unconditional basis
such that the set of Daugavet points is weakly dense
in the unit ball \cite[Theorem~4.7]{ALMT}.
A 1-unconditional basis does however prevent the existence
of super $\Delta$-points \cite[Proposition~2.12]{ALMT}.
The main result in Section~\ref{subsec:unconditional_bases}
is Theorem~\ref{thm:shrinking_RMAP_no_delta}.
The proof of this theorem relies on
Theorem~\ref{thm:D-point/weak*-Delta-points_implies_weak*-super-Delta-point}
and the theorem is used to show that if a Banach space $X$ has
a shrinking $k$-unconditional basis for $k < 2$,
then $X$ contains no $\mathfrak{D}$-points and
$X^*$ contains no weak$^*$ $\Delta$-points
(see Corollary~\ref{cor:uncond_basis_c0l1}).
This result answers \cite[Question~5.6]{ALM} affirmatively and
can be used to strengthen \cite[Proposition~4.6]{ALM}.
If $X$ has a monotone boundedly complete $k$-unconditional basis for $k < 2$,
we only get that $X$ has no $\Delta$-points.
Furthermore Corollary~\ref{cor:refl_ufdd_no_Dpts} says
that if $X$ is a reflexive Banach space
with a $k$-unconditional basis for $k < 2$,
then $X$ and $X^*$ do not contain $\mathfrak{D}$-points.
This strengthens previous results in this direction.

In Section~\ref{subsubsec:Hahn-Banach-smooth-spaces} we study
the implications of
Theorem~\ref{thm:D-point/weak*-Delta-points_implies_weak*-super-Delta-point}
for M-embedded spaces.
A question that has not appeared in print,
but has been in the back of the mind of several people studying
$\Delta$-points, is the following:
Do (non-reflexive) M-embedded Banach spaces and their duals fail
to contain $\Delta$-points?
Recall that $X$ is \emph{M-embedded} if $X$ is an M-ideal
in its bidual, which means that we can write
$X^{***} = X^* \oplus_1 X^\perp$.
Using the renorming of $\ell_2$ from
Section~\ref{sec:superr-banach-space} we can answer
the above question negatively by showing that
there exists a non-reflexive M-embedded Banach space $X$
such that both $X$ and $X^*$ have a super $\Delta$-point.

However, all is not lost.
M-embedded spaces are known to be Hahn--Banach smooth
and Asplund (c.f. \cite[Chapter~III]{HWW}).
We are able to show that if the dual unit ball
is weak$^*$ sequentially compact, then
we get a sequential version of
Theorem~\ref{thm:D-point/weak*-Delta-points_implies_weak*-super-Delta-point}
and that if $X$ contains a $\mathfrak{D}$-point
or if $X^*$ contains weak$^*$ $\Delta$-point,
then $X^*$ fails to be Kadets--Klee.
For Lipschitz-free spaces this implies that
if $M$ is a metric space such that
$\lipfree{M}$ is a dual space and $Y$ is an M-embedded
(or more generally, Hahn--Banach smooth) predual,
then $Y$ contains no $\mathfrak{D}$-points
and $\lipfree{M}$ contains no weak$^*$ $\Delta$-points
(see Corollary~\ref{cor:free-spaces_HBS-preduals}).

Finally, in Section~\ref{sec:metr-char-delta-molecules},
our main goal is to obtain a purely metric characterization of
those molecules $m_{xy}$ that are $\Delta$-points of $\lipfree{M}$,
called simply \emph{$\Delta$-molecules}
(see Theorem~\ref{thm:delta_molecules}).
In particular, we get that, when $M$ is proper, $m_{xy}$ is a
$\Delta$-point if and only if $x$ and $y$ are connected with a
geodesic (see Corollary~\ref{cor:cmsp-no-delta-mpq}).
These results improve Proposition~4.2 and Theorem~4.13
in \cite{JungRueda}, respectively.
We also prove that every $\Delta$-point of $\lipfree{M}$
with finite support is a finite convex sum
of $\Delta$-molecules
(see Theorem~\ref{thm:delta_finite_conv_comb}),
solving \cite[Problem~3]{VeeorgStudia} in the positive.

\subsection*{Notation}

Let $X$ be a Banach space and let $x \in S_X$.
Recall that $x$ is a \emph{denting point} of $B_X$
if for any $\delta > 0$ there exists a slice
$S$ of $B_X$ such that $x \in S$ and $\diam(S) < \delta$.
Furthermore, $x$ is said to be \emph{strongly exposed}
if there exists $x^* \in D(x)$ such that
$\diam S(x^*,\varepsilon) \to 0$ as $\varepsilon \to 0$.
An element $x^* \in B_{X^*}$ is weak$^*$ strongly exposed
if its strongly exposed by some $x \in D(x^*) \cap X$.
Replacing the diameter of the slices with the Kuratowski measure
of non-compactness $\alpha$ we can define the notion of
\emph{$\alpha$-strongly exposed point}.
This is similar to how quasi-denting points were generalized
from denting points by Giles and Moors \cite{MR1188888}.

We will follow standard Banach space notation as found
in the books \cite{AlbiacKalton} and \cite{MR2766381},
but let us also say a few things about Lipschitz-free space notation
since it is not yet completely standard.

For a metric space $(M,d)$ we denote by $B(x,r)$ the closed
ball centered at $x\in M$ with radius $r$, and we define
the \emph{metric segment} between points $x,y\in M$ by
\begin{equation*}
  [x,y] := \set{ z \in M : d(x,z) + d(z,y) = d(x,y)}.
\end{equation*}
If $M$ is \emph{pointed}, that is, it is equipped with
a distinguished base point usually denoted by $0$,
we let $\Lip_0(M)$ be the space
of Lipschitz functions $f : M \to \R$ such that $f(0) = 0$
equipped with the Lipschitz norm
\begin{equation*}
  \lipnorm{f} := \sup \set{
    \frac{|f(x) - f(y)|}{d(x,y)}
    : x, y \in M, x \neq y
  }.
\end{equation*}
Let $\delta\colon M \to \Lip_0(M)^*$ be the map
that assigns each $x \in M$ to the corresponding
point-evaluation $\delta(x)$ in $\Lip_0(M)^*$,
that is $f(x) = \duality{\delta(x), f}$
for $x \in M$ and $f \in \Lip_0(M)$.
It is well-known that $\delta$ is a non-linear isometry
and that $\lipfree{M} = \csp{}(\delta(M))$
is a predual of $\Lip_0(M)$ called the \emph{Lipschitz-free space}
over $M$.
The Lipschitz-free space is also known under the name
Arens--Eells space and the notation $\text{\AE}(M)$ is
sometimes used, e.g. in \cite{Weaver2}.

Recall that a function $f : M \to \R $, where $M$ is a metric space,
is \emph{locally flat} if
\begin{equation*}
  \lim_ {x,y\to z} \frac{f(x) - f(y)}{d(x,y)} = 0
\end{equation*}
for every $z \in M$.
We will follow \cite{AGPP} and define
\begin{equation*}
  \lip_0(M) := \set{ f \in \Lip_0(M) :
    f\ \text{is locally flat and}\ 
    \lim_{r \to \infty} \lipnorm{\restr{f}{M \setminus B(0,r)}} = 0
    }.
\end{equation*}
For compact $M$, $\lip_0(M)$ is simply the set of
locally flat $f:M\to\R$ such that $f(0)=0$.

A \emph{molecule} $m_{xy} \in \lipfree{M}$, $x \neq y$, is
an element of $S_{\lipfree{M}}$ of the form
\begin{equation*}
  m_{xy} := \frac{\delta(x) - \delta(y)}{d(x,y)}.
\end{equation*}
For a subset $K$ of $M$, the Lipschitz-free space
$\lipfree{K\cup\set{0}}$ is identified with the
subspace $\csp{}(\delta(K))$ of $\lipfree{M}$.
The \emph{support} of $\mu \in \lipfree{M}$, denoted $\supp(\mu)$,
is the intersection of all closed $K\subseteq M$ such that
$\mu \in \lipfree{K\cup\set{0}} \subseteq \lipfree{M}$.
It holds that $\mu\in\lipfree{\supp(\mu)\cup\set{0}}$
(see \cite[Section~2]{APPP}).

\section{A separable dual space with a Daugavet point}
\label{sec:separable-dual-space}

This section is dedicated to studying the example
given by Veeorg \cite[Example~3.1]{VeeorgStudia} of a Banach space
with the \RNP whose unit sphere contains a Daugavet point.
We consider a metric space $\mathcal{M}$ constructed from
a subset of $\R^2$ as follows.
Let $p:=(0,0)$, $q:=(1,0)$ and for every $n \in \N$ let
\begin{equation*}
  S_n := \set{(2^{-n}k , 2^{-n}) \, : \, k=0, 1, \ldots, 2^n}
\end{equation*}
and finally $\mathcal{M} := \set{p,q} \cup \bigcup_{n=1}^\infty S_n$.
(See Figure~\ref{figure:TVspace}.)
Endow $\mathcal{M}$ with the metric
\begin{equation*}
  d((x_1,y_1), (x_2,y_2))
  :=
  \begin{cases}
    \abs{x_1-x_2}, & \text{if } y_1 = y_2; \\
    \abs{y_1-y_2} + \min\set{x_1 + x_2, 2 - (x_1 + x_2)},
    & \text{if } y_1 \neq y_2;
  \end{cases}
\end{equation*}
and take $p$ as its base point.

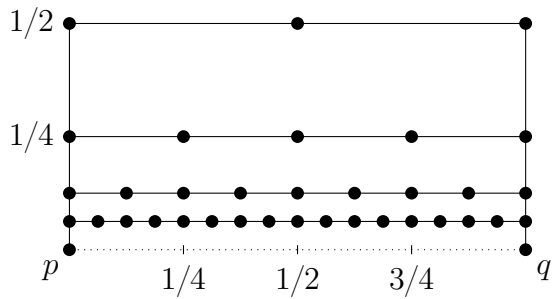
\begin{figure}[ht]
  \centering
  \begin{tikzpicture}[scale=6.0]
    \draw[dotted] (-0.01,0) -- (1.0,0);
    \draw (0,-0.01) -- (0,0.5);
    \draw (1,0) -- (1,1/2);

    \foreach \y in {1/4, 1/2} {
      \draw (0.01,\y) -- (-0.01,\y) node[left] {$\y$};
    }

    \foreach \x in {1/4, 1/2, 3/4} {
      \draw (\x,0.01) -- (\x,-0.01) node[below] {$\x$};
    }

    \fill (0,0) circle[radius=0.4pt] node[below left] {$p$};
    \fill (1,0) circle[radius=0.4pt] node[below right] {$q$};

    \foreach \sn in {1,2,3,4} {
      \draw (0,{2^(-\sn)}) -- (1,{2^(-\sn)});

      \pgfmathtruncatemacro\aa{2^\sn}
      \foreach \k in {0,...,\aa} {
        \fill ({\k*2^(-\sn)},{2^(-\sn)}) circle[radius=0.4pt];
      }
    }
  \end{tikzpicture}
  \caption{The sets $S_0,\ldots,S_4$}
  \label{figure:TVspace}
\end{figure}

Veeorg showed that $\lipfree{\mathcal{M}}$ has the \RNP and
that the molecule $m_{pq}$ is a Daugavet point.
In this section we will show that $\lipfree{\mathcal{M}}$
is isomorphic to $\ell_1$ and that it actually has a predual.
Let us start by introducing our candidate for a predual.

Denote
\begin{equation*}
  V := \set{(x,y) \in \mathcal{M} : x = 0 \text{ or } x = 1}
\end{equation*}
and let $h : \mathcal{M} \to \R$ be the function defined by
\begin{equation*}
  h(x,y) := x .
\end{equation*}
Observe that $h \in S_{\Lip_0(\mathcal{M})}$.
Define
\begin{equation*}
  Y :=
  \set{
    f \in \Lip_0(\mathcal{M}) \,:\,
    \lim_n\lipnorm{\restr{(f-f(q)\cdot h)}{S_n}} \to 0
    \text{ and }
    \restr{f}{V}
    \text{ is locally flat}
  }.
\end{equation*}

We can now state the main theorem of this section.

\begin{thm}\label{thm:veeorg_space}
  Let $\mathcal{M}$ and $Y$ be as defined above.
  Then the following holds:
  \begin{enumerate}
  \item\label{item:thm:veeorg_dual}
    The Banach space $Y$ satisfies
    $Y^* = \lipfree{\mathcal{M}}$;
  \item\label{item:thm:veeorg_l1}
    $\lipfree{\mathcal{M}}$ is isomorphic to $\ell_1$.
  \end{enumerate}
  Thus, there is a separable dual space isomorphic to $\ell_1$
  that admits a Daugavet point.
\end{thm}

\begin{proof}[Proof of Theorem~\ref{thm:veeorg_space}~\ref{item:thm:veeorg_dual}]
  According to a theorem by Petunin and Plichko
  \cite[Theorem~4]{PetuninPlichko},
  given a separable Banach space $X$,
  a subspace $Y$ of $X^\ast$ is an isometric predual of $X$
  when it satisfies the following conditions:
  \begin{enumerate}[label={(\alph*)}]
  \item\label{item:ppcond1}
    $Y$ is norm closed,
  \item\label{item:ppcond2}
    $Y$ separates points of $X$,
  \item\label{item:ppcond3}
    all elements of $Y$ attain their norm on $S_X$.
  \end{enumerate}
  We will apply this result to $X:=\lipfree{\mathcal{M}}$
  and the subspace $Y$ of $\Lip_0(\mathcal{M})$.
  Let us verify that $Y$ satisfies all three Petunin--Plichko conditions.

  We start with \ref{item:ppcond1}.
  Suppose that $(f_n)$ is a sequence in $Y$ that converges
  in norm to $f\in\Lip_0(\mathcal{M})$.
  Fix $\varepsilon > 0$, then there is $k$ such that
  $\lipnorm{f - f_k} < \varepsilon$, and we get for every $n$
  \begin{align*}
    \lipnorm{\restr{(f-f(q)\cdot h)}{S_n}}
    &\leq
      \lipnorm{\restr{(f-f_k)}{S_n}}
      + \lipnorm{\restr{(f_k-f_k(q)\cdot h)}{S_n}}
      + \lipnorm{\restr{((f_k(q)-f(q))\cdot h)}{S_n}} \\
    &<
      2\varepsilon + \lipnorm{\restr{(f_k-f_k(q)\cdot h)}{S_n}}.
  \end{align*}
  Since $f_k\in Y$, this will be less than $3\varepsilon$
  for $n$ large enough.
  Moreover, $V$ is compact and so the space
  $\lip_0(V)$ is closed (see e.g. \cite[Corollary~4.5]{Weaver2}).
  Since $\restr{f_k}{V} \to \restr{f}{V}$ in norm,
  we get $\restr{f}{V}\in\lip_0(V)$ as well, so that $f\in Y$.
  Thus $Y$ is closed.

  Let us now verify \ref{item:ppcond2}.
  Fix $\mu \in \lipfree{\mathcal{M}}$, $\mu \neq 0$.
  Suppose first that $\supp(\mu)$
  contains no isolated point of $\mathcal{M}$.
  Then $\supp(\mu)=\set{q}$ as, by definition,
  the base point cannot be an isolated point
  of $\supp(\mu)$. Thus $\mu$
  is a nonzero multiple of $\delta(q)$,
  so that $h(\mu) \neq 0$ with $h\in Y$.
  Now suppose that $\supp(\mu)$ contains some isolated point
  $x$ of $\mathcal{M}$.
  Then $\set{x}$ is an open neighborhood of $x$,
  and so by \cite[Proposition~2.7]{APPP} there exists
  $f \in \Lip_0(\mathcal{M})$, supported on $\set{x}$,
  such that $f(\mu) \neq 0$.
  In other words, $\chi_{\set{x}}(\mu)\neq 0$,
  where $\chi_{\set{x}}$ is the characteristic function of the set $\set{x}$.
  But $\chi_{\set{x}}\in Y$, so this finishes the proof of \ref{item:ppcond2}.

  Finally we check \ref{item:ppcond3}.
  We will see that, in fact, every $f\in Y$ attains its Lipschitz
  constant between two points of $\mathcal{M}$, and thus it attains
  its norm as a functional at some molecule.
  Fix $f\in Y$ and assume that $\lipnorm{f}=1$.
  Suppose that $f$ does not attain its Lipschitz constant.
  Then we may still find a sequence of pairs of points $(u_n,v_n)$ in
  $\mathcal{M}$ such that $f(m_{u_nv_n}) \to 1$.
  Note that $x\in [u,v]$ implies that $m_{uv}$
  is a convex combination of $m_{ux}$ and $m_{xv}$, hence
  \begin{equation*}
    \max \set{f(m_{ux}), f(m_{xv})}
    \geq f(m_{uv}).
  \end{equation*}
  Thus, by replacing each pair $u_n, v_n$ with other points
  in $[u_n,v_n]$ and passing to a subsequence if necessary,
  we may assume that either $u_n,v_n\in V$ for all $n$
  or there is a sequence $(k_n)$ in $\N$ such that
  $u_n, v_n \in S_{k_n}$ for all $n$.

  In the first case, all $u_n,v_n$ belong to the compact set $V$, so
  by passing to a subsequence we get $u_n\to u$, $v_n\to v$ for some
  $u,v\in V$.
  Note that it is impossible to get $u=v$, as that would contradict
  the fact that $\restr{f}{V}$ is locally flat.
  Thus $u\neq v$ and so $f(m_{uv})=1$.

  Now suppose that $u_n,v_n\in S_{k_n}$ for all $n$.
  If $(k_n)$ is bounded, say by $N$, then this implies that $f$
  restricted to the finite set $S_1 \cup \cdots \cup S_N$
  has Lipschitz constant $1$, and so it must attain its
  Lipschitz constant in that set.
  Otherwise we may assume that $k_n\to\infty$.
  Then we have $\lipnorm{\restr{f}{S_{k_n}}} \geq f(m_{u_nv_n})$ and,
  given that $\lipnorm{\restr{f}{S_{k_n}}} \to \abs{f(q)}$,
  we obtain $\abs{f(q)}=1$.
  So $f$ attains its Lipschitz constant between $p$ and $q$.

  Thus, all conditions in the Petunin--Plichko theorem are satisfied, and
  this finishes the proof that $Y^\ast=\lipfree{\mathcal{M}}$.
\end{proof}

\begin{rem}
  The definition of $Y$ is not equivalent if we ask that $f$ is
  locally flat instead of $\restr{f}{V}$. For instance, $h$ is not locally
  flat while $\restr{h}{V}$ is. It is conjectured that whenever a
  Lipschitz-free space is a separable dual, it must admit a predual that
  consists entirely of locally flat functions. Preduals of
  Lipschitz-free spaces are not unique in general, so there might be an
  alternative predual $Y$ satisfying that condition.
\end{rem}

\begin{rem}\label{rem:exmp42_ANPP_is_dual}
  The argument in the proof of
  Theorem~\ref{thm:veeorg_space}~\ref{item:thm:veeorg_dual}
  can be easily adapted to show that the Lipschitz-free space from
  \cite[Example~4.2]{ANPP} is also a dual space.
  The corresponding predual has a simpler description,
  as the local flatness condition can be dropped.
  Note that the molecule $m_{0q}$ in that example is a $\Delta$-point
  (this can be shown e.g. by using Theorem~\ref{thm:delta_molecules}
  below) but not a Daugavet-point as there are denting points
  in the unit ball of this space at distance strictly less than 2 to it.
\end{rem}

Before we engage with the proof of
Theorem~\ref{thm:veeorg_space}~\ref{item:thm:veeorg_l1}
let us note that the proof below
can be adapted to show that the Lipschitz-free space over other similar
metric spaces, such as \cite[Examples 4.2 and 4.3]{ANPP},
is also isomorphic to $\ell_1$.
It is based on the next general lemma, which can be understood
as a finite version of the approach followed in \cite{AACD_sums}.

\begin{lem}\label{lm:finite_decomp}
  Let $M$ be a complete pointed metric space,
  and $\varphi_1,\ldots,\varphi_n$ be non-negative Lipschitz functions
  on $M$ with bounded support and such that
  $\varphi_1 + \cdots + \varphi_n = 1$.
  Suppose that $A_1,\ldots,A_n$ are subsets of $M$ containing
  the base point, and $\supp(\varphi_k)\subset A_k$ for all $k$.
  Then $\lipfree{M}$ is isomorphic to a complemented
  subspace of $\lipfree{A_1} \oplus \cdots \oplus \lipfree{A_n}$.
\end{lem}

\begin{proof}
  For each $k=1,\ldots,n$ let $W_k \colon \lipfree{M} \to \lipfree{M}$
  be the weighting operator defined by
  \begin{equation*}
    \duality{W_k \mu,f} := \duality{\mu, f \cdot \varphi_k}
  \end{equation*}
  for $\mu \in \lipfree{M}$ and $f \in \Lip_0(M)$.
  By the results in \cite[Section~2]{APPP},
  this is a well-defined bounded operator.
  Moreover, its range is contained in $\lipfree{A_k}$,
  which we identify with the corresponding subspace of
  $\lipfree{M}$.
  Now define operators
  \begin{align*}
    T &\colon
        \lipfree{M} \to
        \lipfree{A_1} \oplus \cdots \oplus\lipfree{A_n} \\
    S &\colon
        \lipfree{A_1} \oplus \cdots \oplus \lipfree{A_n}
        \to
        \lipfree{M}
  \end{align*}
  by $T\mu := (W_1 \mu,\ldots,W_n \mu)$ and
  $S(\mu_1, \ldots, \mu_n) := \mu_1 + \cdots + \mu_n$.
  Both of them are clearly bounded, and
  for every $\mu\in\lipfree{M}$ and $f\in\Lip_0(M)$ we have
  \begin{align*}
    \duality{ST\mu,f}
    &= \duality{W_1 \mu,f} + \cdots + \duality{W_n \mu,f} \\
    &= \duality{\mu,f\cdot\varphi_1} + \cdots + \duality{\mu,f\cdot\varphi_n} \\
    &= \duality{\mu,f}
  \end{align*}
  by the choice of $\varphi_k$.
  Thus $ST$ is the identity on $\lipfree{M}$.
  Therefore $P := TS$ is a projection of
  $\lipfree{A_1} \oplus \cdots \oplus \lipfree{A_n}$
  onto its subspace $T(\lipfree{M})$,
  which is isomorphic to $\lipfree{M}$.
\end{proof}

The following provides a simple sufficient condition allowing
Lemma~\ref{lm:finite_decomp} to be applied.

\begin{lem}\label{lm:dist_suff_cond}
  Let $M$ be a bounded complete metric space,
  and let $U_1,\ldots,U_n$ be an open cover of $M$.
  Suppose that
  \begin{equation*}
    \inf_{x \in M} \sum_{k=1}^n d(x, M\setminus U_k) > 0.
  \end{equation*}
  Then there exist $\varphi_1,\ldots,\varphi_n$,
  non-negative Lipschitz functions on $M$,
  such that $\varphi_1 + \cdots + \varphi_n = 1$
  and each $\varphi_k$ vanishes outside of $U_k$.
\end{lem}

\begin{proof}
  It is enough to take
  \begin{equation*}
    \varphi_k(x)
    :=
    \frac{d(x,M\setminus U_k)}{\sum_{i=1}^n d(x,M\setminus U_i)}
  \end{equation*}
  for $x\in M$.
  The denominator is Lipschitz and bounded below by assumption,
  so each $\varphi_k$ is Lipschitz
  (see e.g. \cite[Proposition 1.30]{Weaver2}),
  and the other conditions are satisfied trivially.
\end{proof}

\begin{proof}[Proof of Theorem~\ref{thm:veeorg_space}~\ref{item:thm:veeorg_l1}]
  Fix real numbers $\alpha, \beta$ with
  $0 < \beta < \alpha < \frac{1}{2}$
  and consider the sets
  \begin{align*}
    A &:= \set{(x,y) \in \mathcal{M} \,:\, x<\alpha} \\
    B &:= \set{(x,y) \in \mathcal{M} \,:\, \beta<x<1-\beta} \\
    C &:= \set{(x,y) \in \mathcal{M} \,:\, x>1-\alpha}
  \end{align*}
  which form an open cover of $\mathcal{M}$ to which we wish to
  apply Lemma~\ref{lm:dist_suff_cond}.
  Note that $A \cap C = \varnothing$ while $B$
  intersects both $A$ and $C$.
  For $z=(x,y)\in\mathcal{M}$ denote
  \begin{equation*}
    D(z)
    :=
    d(z, \mathcal{M}\setminus A)
    + d(z,\mathcal{M}\setminus B)
    + d(z,\mathcal{M}\setminus C).
  \end{equation*}
  Let us see that $D(z)\geq\alpha-\beta$ for all $z$.
  By symmetry, it is enough to verify this when $x \leq \frac{1}{2}$.
  Then we have three cases:
  \begin{itemize}
  \item
    If $z \in A \setminus B$ then
    $D(z) = d(z, \mathcal{M} \setminus A)$.
    So either $z = p$ and then $D(z) \geq \alpha$,
    or $x \leq \beta$ and $y > 0$,
    then the closest point to $z$ in $\mathcal{M} \setminus A$
    has the form $(x_a, y)$ for $x_a\geq\alpha$,
    and so $D(z) = x_a - x \geq \alpha - \beta$.
  \item
    If $z \in A \cap B$ then
    $D(z)
    = d(z,\mathcal{M}\setminus A) + d(z,\mathcal{M}\setminus B)$,
    and we have $y > 0$ and $\beta < x < \alpha$.
    Thus the closest points in $\mathcal{M} \setminus A$ and
    $\mathcal{M} \setminus B$ are $(x_a, y)$ and $(x_b, y)$
    where $x_a \geq \alpha$ and $x_b \leq \beta$,
    respectively, so $D(z) = (x_a-x) + (x-x_b) \geq \alpha - \beta$.
  \item
    If $z \in B \setminus (A\cup C)$
    then $D(z) = d(z,\mathcal{M}\setminus B)$,
    and we have $y > 0$ and $\alpha \leq x \leq\frac{1}{2}$.
    Thus the closest point in $\mathcal{M}\setminus B$ is
    $(x_b,y)$ where $x_b \leq \beta$, and
    $D(z) = x - x_b \geq \alpha - \beta$.
  \end{itemize}
  Thus $\inf D(z)>0$, and we may apply Lemma~\ref{lm:dist_suff_cond}
  followed by Lemma~\ref{lm:finite_decomp} to conclude that
  $\lipfree{\mathcal{M}}$ is isomorphic to a complemented subspace of
  \begin{equation*}
    \lipfree{\cl{A}} \oplus \lipfree{\cl{B} \cup \set{p}}
    \oplus \lipfree{\cl{C} \cup \set{p}}.
  \end{equation*}
  We will prove that each of these three Lipschitz-free spaces
  is isomorphic to $\ell_1$, and then the result will follow
  by Pe{\l}czy\'nski's classical theorem that complemented
  infinite-dimensional subspaces of $\ell_1$ are isomorphic
  to $\ell_1$.

  Note that the set $\cl{A}$ is an infinite weighted tree, i.e.,
  a connected graph with no cycle.
  In particular, it is isometric to a subset of an $\R$-tree
  that contains all of its branching points
  $(0, 2^{-n})$, $n \in \N$.
  Thus $\lipfree{\cl{A}}$ is isometric to $\ell_1$ by
  \cite[Corollary~3.4]{Godard}.

  Let $K_0 := \set{p}$ and
  $K_n := \set{(x,y) \in \cl{B} : y = 2^{-n}}$ for each $n \in \N$.
  Then each $\lipfree{K_n}$ is isometric to a finite dimensional
  $\ell_1$-space, and there is a bound above and below on the
  distance between elements in distinct $K_n$'s.
  Using  \cite[Proposition~5.1]{Godard} we get that
  $\lipfree{\cl{B} \cup \set{p}}$ is isomorphic to $\ell_1$.

  Finally, notice that $\lipfree{\cl{C}} = \lipfree{\cl{A}}$
  as $\cl{C}$ and $\cl{A}$ are isometric.
  On the other hand, $\lipfree{\cl{C}}$ and
  $\lipfree{\cl{C}\cup\set{p}}$
  are isomorphic by \cite[Lemma~2.8]{AACD_sums}.
  Thus $\lipfree{\cl{C}\cup\set{p}}$ is also
  isomorphic to $\ell_1$ and this ends the proof.
\end{proof}

We end this section by showing a couple of geometric
properties of the predual $Y$ of $\lipfree{\mathcal{M}}$.
Recall that a Banach space $X$ is \emph{almost square}
if for every finite subset $x_1,\ldots,x_n \in S_X$
and $\varepsilon > 0$, there exists $y \in S_X$ such that
$\|x_i \pm y\| \le 1 + \varepsilon$ for $i = 1, \ldots, n$
\cite{MR3415738}.
The space $c_0$ is the prototypical almost square Banach space.

For a proper metric space $M$ we have that
$\lip_0(M)$ is, for any $\varepsilon > 0$,
$(1 + \varepsilon)$-isometric to a subspace of $c_0$
by \cite[Lemma~3.9]{MR3376824}.
Hence $\lip_0(M)$ is almost square \cite[Example~3.2]{MR3415738}.
Next we show that the predual of $\lipfree{\mathcal{M}}$
shares this property.
At the end of Section~\ref{subsubsec:Hahn-Banach-smooth-spaces},
we will see that, unlike $\lip_0(M)$ for $M$ proper
and purely 1-unrectifiable, $Y$ is not M-embedded.

\begin{prop}\label{prop:FV_almost_square}
  The predual $Y$ of the space $\lipfree{\mathcal{M}}$
  is almost square.
\end{prop}

\begin{proof}
  Let $(f_i)_{i = 1}^N \subset S_Y$ and $\eps > 0.$ As
  $\restr{f_i}{V}$ is locally flat we can choose $\delta > 0$ such
  that for all $a, b \in B(q, \delta) \cap V$
  \begin{equation*}
    |\langle f_i, m_{ab} \rangle| < \eps.
  \end{equation*}
  for all $i = 1, \ldots, N$.
  Choose $k$ such that
  $2^{-k} < \eps$ and such that the points
  $(1,2^{-k+1})$, $(1,2^{-k})$ and
  $(1,2^{-k-1})$ are all in $B(q,\delta)$.
  Define $g \in Y$ by
  \begin{equation*}
    g(a) :=
    \begin{cases}
      2^{-(k+1)} h(a), & a \in S_k \\
      0, &\text{otherwise}.
    \end{cases}
  \end{equation*}
  Let $a_0 = (1,2^{-k}) \in S_k$. It is clear that $\lipnorm{g} = 1$,
  since the closest point to $a_0$ is at distance $2^{-(k+1)}$ and
  $h(a_0) = 1$.
  Let us check that $\lipnorm{f_i \pm g} \le 1 + \eps$.

  If $a,b \notin \supp(g) \subset S_k$, then
  \begin{equation*}
    |\langle f_i \pm g, m_{ab} \rangle|
    =
    |\langle f_i, m_{ab} \rangle| \le 1.
  \end{equation*}
  If $a,b \in S_k$, then
  \begin{equation*}
    |\langle f_i \pm g, m_{ab} \rangle|
    \le
    |\langle f_i, m_{ab} \rangle|
    +
    2^{-(k+1)}
    |\langle h, m_{ab} \rangle|
    \le 1 + \eps.
  \end{equation*}
  If $a \in S_k$ and $b \notin S_k$,
  then we can find $a' \in V \cap S_k$
  and $b' \in V$ (at the level of $b$)
  such that $a' \in [a,b]$ and $b' \in [a',b]$.
  Now $m_{ab}$ is a convex combination
  of $m_{aa'}$ and $m_{a'b}$, and
  $m_{a'b}$ is a convex combination
  of $m_{a'b'}$ and $m_{b'b}$.
  Hence
  \begin{align*}
    |\langle f_i \pm g, m_{ab}\rangle|
    &\le
    \max\{
    |\langle f_i \pm g, m_{aa'}\rangle|,
    |\langle f_i \pm g, m_{a'b'}\rangle|,
    |\langle f_i \pm g, m_{b'b}\rangle|
    \} \\
    &\le
    \max\{
    1 + \eps,
    |\langle f_i \pm g, m_{a'b'}\rangle|
    \}.
  \end{align*}
  If $a'$ and $b'$ have first coordinate $0$, then
  we are done by the above.
  If $a'$ and $b'$ have first coordinate $1$, then
  by the convex combination trick we may assume
  that $b' \in B(q,\delta)$
  (e.g. $b' = (1,2^{-k+1})$ or $b' = (1,2^{-k-1})$).
  But for $a', b' \in B(q,\delta)$ we have
  \begin{equation*}
    |\langle f_i \pm g, m_{a'b'} \rangle|
    \le
    |\langle f_i, m_{a'b'} \rangle|
    +
    |\langle g, m_{a'b'} \rangle|
    \le
    \eps + 1
  \end{equation*}
  as desired.
\end{proof}

So while our predual $Y$ shares some properties with $c_0$
we will now give a short proof that, unlike $c_0$,
$Y$ is not polyhedral.
Recall that a Banach space $X$ is \emph{polyhedral} is the
unit ball of every finite dimensional subspace of $X$ is a polytope.

\begin{prop}\label{prop:FV_not_polyhedral}
  The predual $Y$ of $\lipfree{\mathcal M}$ is not polyhedral.
\end{prop}

\begin{proof}
  It suffices to prove that the set of $E$ of weak$^*$ strongly exposed
  points of $B_{\lipfree{\mathcal M}}$ is not a boundary for $Y$
  (c.f. e.g. Theorem~1.4 in \cite{MR1792984}).
  Now every weak$^*$ strongly exposed point is also a preserved
  extreme point and so, by e.g. \cite[Corollary~3.44]{Weaver2} and
  \cite[Theorem~1.1]{AP20},
  % Alternative: Lemma~5.1 and Theorem~5.4 in GLPRZ or Main Theorem in AG19
  for every $\mu \in E$ there are $a,b \in \mathcal{M}$,
  $a \neq b$, with $[a,b] = \set{a,b}$ such that $\mu = m_{ab}$.
  Now, define the function $f: \mathcal{M} \to \R$ by
  \begin{equation*}
    f(x,y) := x(1 - y^2)
  \end{equation*}
  for $z = (x,y) \in \mathcal{M}$.
  It is easy to check that $f \in S_Y$.
  % vl: Left details here, if they are ever needed.
  % We have $|\duality{f,m_{pq}}| = 1$, so $\lipnorm{f} \ge 1$.
  % Note that $f(0,y) = 0,$
  % $f$ is linear on $S_n$ with slope $1 - y^2 < 1,$
  % and that the largest slope on $V$ is between
  % $a = (1,\frac{1}{4})$ and $b = (1,\frac{1}{2})$
  % with $\duality{f, m_{ab}} = \frac{3}{4}$.
  % Hence $\lipnorm{f} = 1$.
  % Since $\frac{\partial f}{\partial y} = - 2y$
  % and $\frac{\partial f}{\partial y}|_{(1,0)} = 0,$
  % $f$ is locally flat on $V,$ so $f \in S_Y.$

  By construction, $|\duality{f,m_{ab}}| < 1$ for all
  $a, b$ with $[a,b] = \set{a,b},$
  except for $|\duality{f,m_{pq}}| = 1$.
  But $m_{pq}$ is a Daugavet point,
  so $E$ is not a boundary,
  and thus $Y$ is not polyhedral.
\end{proof}

\section{A superreflexive Banach space with a
  \texorpdfstring{$\Delta$}{Delta}-point}
\label{sec:superr-banach-space}

In \cite{DKR+} a renorming of $\ell_2$ was used to show
that there exists a Banach space whose norm is
asymptotically midpoint uniformly convex, but
not asymptotically uniformly convex.
We will use a slight variation of this norm to answer
Question~6.1 from \cite{ALMP} negatively;
there exists an equivalent renorming of $\ell_2$ with
a $\Delta$-point.
We will also show that this $\Delta$-point fails to be a Daugavet point
in a very strong way by providing a sequence of strongly exposed points
of the new unit ball that converges in norm to it. It is still open whether there
exists a reflexive or superreflexive space with a Daugavet point.

Let $(e_n)_{n \geq 1}$ be the usual basis in $\ell_2$ and denote
the biorthogonal elements in $\ell_2$ by $(e_n^*)_{n \geq 1}$.
We follow \cite{DKR+} and introduce an equivalent
norm on $\ell_2$ by defining the unit ball by
\begin{equation*}
  B(\ell_2,\|\cdot\|)
  :=
  \clco\left(B_{\ell_2} \cup \set{ \pm (e_1 + e_n)_{n \ge 2}} \right)
\end{equation*}
where we take the closure of the convex hull in the topology of
$\|\cdot\|_2$.
From \cite[Lemma~2.5]{DKR+} we have that
$\|\cdot\| \le \|\cdot\|_2 \le \sqrt{2}\|\cdot\|$,
that $(e_n)_{n \geq 1}$ is a normalized bimonotone basis
for $(\ell_2,\|\cdot\|)$, and that
$\|e_1 + e_n\| = 1$ for $n \ge 2$.

We trim down the $\|\cdot\|$ norm and define
for $x = \sum_{n=1}^\infty x_n e_n \in \ell_2$
\begin{equation*}
  \nnorm{x} := \max\{\|x\|, \sup_{n \ge 2}|x_1 - 2x_n|\},
\end{equation*}
and $Y := (\ell_2, \nnorm{\cdot}).$
We have $\|y\| \le \nnorm{y} \le 3\|y\|$ for all $y \in Y$.
Figure~\ref{fig:renorml2} shows a picture in $\linsp(e_1,e_n)$ of
$S_{(\ell_2,\|\cdot\|)}$ in red and $S_Y$ in blue.
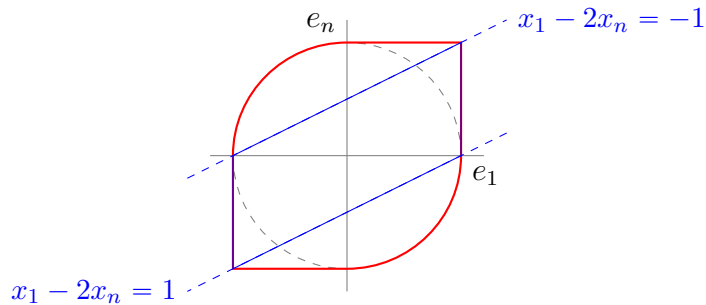
\begin{figure}[ht]
  \centering
  \begin{tikzpicture}[scale=1.5]
    \draw[gray] (-1.2,0) -- (1.2,0); %x - axis
    \draw[gray] (0,-1.2) -- (0,1.2); %y - axis
    \draw (1,0) node[below right] {$e_1$};
    \draw (0,1) node[above left] {$e_n$};
    \draw[gray, dashed] (1,0) arc(0:90:1);
    \draw[gray, dashed] (-1,0) arc(180:270:1);

    \draw[red,thick] (1,0) -- (1,1) -- (0,1) arc(90:180:1)
    -- (-1,-1) -- (0,-1) arc(270:360:1);

    \draw[blue, dashed] (-1.4,-0.2) -- (1.4,1.2)
    node[right] {\small $x_1 - 2 x_n = -1$};
    \draw[blue, dashed] (1.4,0.2) -- (-1.4,-1.2)
    node[left] {\small $x_1 - 2 x_n = 1$};

    \draw[blue] (-1,-1) -- (1,0) -- (1,1) -- (-1,0) -- (-1,-1);
  \end{tikzpicture}
  \caption{Geometric idea of the renorming}
  \label{fig:renorml2}
\end{figure}
It is clear that $\nnorm{e_1} = 1$, $\nnorm{e_n} = 2$
and $\nnorm{e_1 + e_n} = 1$ for $n \ge 2$.
These are all the ingredients needed to show that
$e_1$ is a $\Delta$-point.

To prove that $e_1$ is not a Daugavet point we will find
strongly exposed points at distance strictly less than $2$ from $e_1$.
As previously mentioned, we will actually prove something more,
as we will show that $e_1$ belongs to the closure
of the set of all strongly exposed points of $B_Y$.
This will be done in the lemmas that follow
the main theorem of this section that we will now state.

\begin{thm}\label{thm:l2_renorm_e1_delta_not_daugavet}
  Let $Y$ be the renorming of $\ell_2$ above.
  We have that both $e_1 \in S_Y$ and $e_1^* \in S_{Y^*}$
  are super $\Delta$-points, but neither of them is a Daugavet point.
\end{thm}

The proof of the theorem uses a few simple lemmas.
In \cite{DKR+}, explicit expressions for how
to calculate the $\|\cdot\|$-norm and its dual
were given. These will turn out to be very useful
even for the $\nnorm{\cdot}$-norm, so let us provide 
right away some more detail.

Recall that in any given vector space $V$, we have
\begin{equation*}
  \conv(A\cup B)
  =
  \{\lambda a+(1-\lambda)b:\ a\in A,\ b\in B,\ \lambda\in[0,1]\}
\end{equation*}
for every convex sets $A, B\subset V$.  In particular, $B(\ell_2,\norm{\cdot})$ is equal to the closure 
of the set
\begin{equation}\label{eqt:dense-set-renorming}
  C :=
  \{\lambda y+(1-\lambda)u :
  y \in B_{\ell_2},\ u\in \conv\{\pm(e_1+e_n),
  \ n\geq 2\},\ \lambda\in[0,1]\}.
\end{equation}
The above description of the unit ball of $(\ell_2,\|\cdot\|)$
is used in \cite[Lemma~2.5~(b)]{DKR+} to show that the norm
of $x^* = \sum_{n=1}^\infty a_n e_n^*$ in the dual is given by
\begin{equation*}
  \|x^*\| = \max \{ \|x^*\|_2, \sup_{n \ge 2} |a_1 + a_n|\}.
\end{equation*}
A more or less identical argument shows that
\begin{equation*}
  B_{Y^*} = \clco \set{
    B_{(\ell_2,\|\cdot\|)^*} \cup \{ \pm (e_1^* -  2e_n^*)_{n \ge 2}\}
  },
\end{equation*}
hence $B_{Y^*}$ is equal to the closure of the set 
\begin{equation}\label{eqt:dense-set-dual-renorming}
  D := \{\lambda y^*+(1-\lambda)u^* :
  \ y^*\in B_{(\ell_2,\|\cdot\|)^*},
  \ u^*\in \conv\{\pm(e_1^*-2e_n^*),
  \ n\geq 2\},\ \lambda\in[0,1]\}.
\end{equation}

Next, let us identify some strongly exposed points of $B_Y$
near $e_1$ and in $B_{Y^*}$ near $e_1^*$.

\begin{lem}\label{lem:seq_conv_norm_e1}
  For each $n \in \N$ define $x := x(n) \in Y$ and
  $x^* := x^*(n) \in Y^*$ by letting $k := 32n-16$ and setting
  \begin{equation*}
    x := \left(1-\frac{1}{n}\right)e_1 + \frac{1}{4n} \sum_{i=2}^{k+1} e_i
    \in Y
  \end{equation*}
  and
  \begin{equation*}
    x^* :=
    \left(1-\frac{1}{n}\right)e^*_1 + \frac{1}{4n} \sum_{i=2}^{k+1} e^*_i
    \in Y^*.
  \end{equation*}
  We have $\nnorm{x} = \|x\| = \|x\|_2 = x^*(x) = 1$,
  and $\nnorm{x^*} = \|x^*\| = \|x^*\|_2 = 1$ for each $n \in \N$.

  Furthermore,
  $\nnorm{e_1 - x} \underset{n}{\to} 0$
  and $\nnorm{e_1^* - x^*} \underset{n}{\to} 0$.
\end{lem}

\begin{proof}
  We have
  \begin{equation*}
    x^*(x) = \|x\|^2_2 = \left(1-\frac{1}{n}\right)^2 + k \cdot \frac{1}{16 n^2}
    = 1 - \frac{2}{n} + \frac{1}{n^2} + \frac{2n-1}{n^2}
    = 1.
  \end{equation*}
  As for the dual we always have $\nnorm{x^*} \le \|x^*\|$
  and if we write $x^* = \sum_{i=1}^\infty a_i e_i^*$, then
  by Lemma~2.5~(b) in \cite{DKR+},
  \begin{equation*}
    \|x^*\| = \max\set{ \|x^*\|_2,\sup_{i\ge 2} |a_1 + a_i|}
    = \max\set{1,1-\frac{3}{4n} } = 1.
  \end{equation*}
  Since $|(1-\frac{1}{n})-\frac{2}{4n}| < 1$ we
  also have $\nnorm{x} = \|x\| \le \|x\|_2 = 1$
  (Lemma~2.5~(a) in \cite{DKR+}).
  Hence $\nnorm{x} = \nnorm{x^*} = 1$.

  Finally
  \begin{equation*}
    \norm{ e_1 - x }_2
    =
\norm{\frac{1}{n}e_1-\frac{1}{4n}\sum_{i=2}^{k+1}e_i}_2= \frac{\sqrt{16+k}}{4n} =\sqrt{\frac{2}{n}}
  \end{equation*}
  and this expression tends to $0$. Hence $\nnorm{e_1-x}\to_n 0$. 
  The calculation for $\nnorm{e_1^* - x^*} \to_n 0$ is similar.
\end{proof}

\begin{lem}\label{lem:xn_close_to_e1_strexp}
  Let $n \in \mathbb{N}$.
  With $x := x(n) \in S_Y$ and $x^* := x^*(n) \in S_{Y^*}$ as in the
  previous lemma we have that $x$ is strongly exposed by $x^*$.
\end{lem}

\begin{proof}  
  Recall that $\|x^*\| = \nnorm{x^*} = 1$.
  The slice $S(B_{\nnorm{\cdot}}, x^*, \delta)$
  is contained in $S(B_{\| \cdot \|}, x^*,\delta)$,
  so it is enough to show that
  $\nnorm{ x - z }\to 0$ uniformly on $S(B_{\| \cdot \|}, x^*, \delta)
    \cap C$ as $\delta\to 0$, where $C$ is the set given by \eqref{eqt:dense-set-renorming},
  and whose closure is equal to $B(\ell_2,\norm{\cdot})$.

  Let $\varepsilon > 0$.
  By uniform convexity there exists $\delta > 0$
  such that $\|x - y\|_2 < \varepsilon$
  whenever $\|x + y\|_2 > 2 - \delta$
  and $\|x\|_2, \|y\|_2 \le 1$.
  We may assume that $\delta < \min \set{\varepsilon,\frac{3}{4n}}$.
  Let $z := \lambda y + (1-\lambda)u$, where $0 \le \lambda \le 1$,
  $y \in B_{\ell_2}$ and $u \in \conv\{\pm (e_1 + e_i)_{i \ge 2}\}$.
  We have
  \begin{equation*}
    |x^*(u)| \le \sup_{i \ge 2} |x^*(e_1 + e_i)| \le 1 - \frac{3}{4n}
    < 1 - \delta.
  \end{equation*}
  So if $x^*(z) > 1 - \delta$, this implies that $x^*(y) > 1 - \delta$.
  Hence
  \begin{equation*}
    \|x + y\|_2 \ge x^*(x + y) > 2 - \delta
  \end{equation*}
  and thus $\|x - y\|_2 < \varepsilon$.
  We also get
  \begin{equation*}
    1 - \delta < \lambda x^*(y) + (1-\lambda)x^*(u)
    \le \lambda + \left(1-\frac{3}{4n}\right)(1-\lambda)
    = 1 - \frac{3}{4n}(1 - \lambda)
  \end{equation*}
  so that $1 - \lambda < \frac{4n}{3}\delta$.
  Now
  \begin{align*}
    \frac{1}{3}\nnorm{ x - z }
    &\le
    \| x - z \|
    \le
    \| x - y\| + \|y - \lambda y\| + \|(1-\lambda)u\|
    \\
    &\le
    \sqrt{2}\|x - y\|_2 + (1-\lambda) + (1-\lambda)
    <
    2\varepsilon + \frac{8n}{3}\delta
    \le
    \left(2 + \frac{8n}{3}\right)\varepsilon,
  \end{align*}
  and since $n$ is a fixed constant, the conclusion follows. 
\end{proof}

\begin{lem}\label{lem:xn_close_to_e1_strexp_dual}
  Let $n \in \mathbb{N}$.
  With $x := x(n) \in S_Y$ and $x^* := x^*(n) \in S_{Y^*}$ as
  in Lemma~\ref{lem:seq_conv_norm_e1}
  we have that $x^*$ is strongly exposed by $x$.
\end{lem}

\begin{proof}
  Again, it is enough to show that $\nnorm{ x^* - z^* } \to 0$ uniformly on 
  $S(x, \delta) \cap D$ as $\delta\to 0$,
  where $D$ is the set given by \eqref{eqt:dense-set-dual-renorming},
  and whose closure is equal to $B_{Y^*}$.
  Let $\varepsilon > 0$.
  By uniform convexity there exists $\delta > 0$
  such that $\|u - v\|_2 < \varepsilon$
  whenever $\|u + v\|_2 > 2 - \delta$
  and $\|u\|_2, \|v\|_2 \le 1$.
  We may assume that $\delta < \min(\varepsilon,\frac{1}{n})$.
  Let $z^* := \lambda y^* + (1-\lambda)u^*$, where $0 \le \lambda \le 1$,
  $y^* \in B_{(\ell_2,\|\cdot\|)^*}$ and
  $u^* \in \conv\{\pm (e_1^* - 2e_i^*)_{i \ge 2}\}$. We have
  \begin{equation*}
    |u^*(x)| \le \sup_{i \ge 2} |(e_1^* - 2e_i^*)(x)|
    \le 1 - \frac{1}{n}
    < 1 - \delta.
  \end{equation*}
  So if $z^*(x) > 1 - \delta$, this implies that $y^*(x) > 1 - \delta$.
  Note that $\|y^*\|_2 \le \|y^*\| \le 1$ and $\|x^*\|_2 = \|x^*\| = 1$.
  Hence
  \begin{equation*}
    \|y^* + x^*\|_2 \ge (x^*+y^*)(x) > 2 - \delta
  \end{equation*}
  and thus $\|y^* - x^*\|_2 < \varepsilon$.
  But then $\|y^* - x^*\| \le \sqrt{2} \|y^*-x^*\|_2 \le \sqrt{2}\varepsilon$.
  We also get
  \begin{equation*}
    1 - \delta < \lambda y^*(x) + (1-\lambda)u^*(x)
    \le \lambda + \left(1-\frac{1}{n}\right)(1-\lambda)
    = 1 - \frac{1}{n}(1 - \lambda)
  \end{equation*}
  so that $1 - \lambda < n \delta$.
  Finally,
  \begin{align*}
    \nnorm{ x^* - z^* }
    &\le
    \nnorm{ x^* - y^*}
    +
    (1-\lambda)
    \nnorm{ y^*}
    +
    (1-\lambda)
    \nnorm{ u^* }
    \\
    &\le
    \| x^* - y^* \|
    +
    2(1-\lambda)
    \\
    &\le
    \sqrt{2}\varepsilon + 2n\delta
    \le
    \left(\sqrt{2} + 2n \right)\varepsilon,
  \end{align*}
   and we are done.
\end{proof}

We are now ready to prove the main theorem of this section.

\begin{proof}[Proof of Theorem~\ref{thm:l2_renorm_e1_delta_not_daugavet}]
  By Lemmas~\ref{lem:seq_conv_norm_e1}
  and \ref{lem:xn_close_to_e1_strexp}
  there exists a sequence $(x(n))_{n \geq 1}$
  of strongly exposed points in $S_Y$ converging
  to $e_1$ in norm, so $e_1$ is clearly not a Daugavet point.
  Similarly, by Lemmas~\ref{lem:seq_conv_norm_e1}
  and \ref{lem:xn_close_to_e1_strexp_dual}
  there exists a sequence $(x^*(n))_{n \geq 1}$
  of strongly exposed points in $S_{Y^*}$ converging
  to $e_1^*$ in norm, so $e_1^*$ is clearly not a Daugavet point.

  It is quite obvious that $x\in S_X$ is a super
  $\Delta$-point if and only if there is a net $(x_\alpha)_{\alpha\in\mathcal{A}}\subset S_X$
  such that $x_\alpha\to x$ weakly and $\norm{x-x_\alpha}\to 2$.
  This was observed in \cite[Proposition~3.4]{MPRZ}. Moreover, if 
  $X^*$ is separable, then we can clearly replace the net with a sequences in
  this characterization.
  Thus, taking $x = e_1$ and $x_n = e_1 + e_n$, $n \ge 2$,
  it follows that $e_1$ is a super $\Delta$-point.

  Finally, let us show that $e_1^*$ is a super $\Delta$-point.
  We have $\nnorm{x^*} \le \|x^*\|$ for all $x^* \in Y^*$.
  In particular, $\nnorm{e_n^*} \le \|e_n^*\| = 1$ for all $n \in \N$.
  For $e_1^*$ we have $e_1^*(e_1) = 1$ and for
  $n \ge 2$ we have and $e_n^*(e_1 + e_n) = 1$
  so $e_n^* \in S_{Y^*}$ for all $n \in \N$.
  Next, we have for $x \in Y$ that
  \begin{equation*}
    |(e_1^* - 2e_n^*)(x)|
    =
    |x_1 - 2x_n| \le \nnorm{x}
  \end{equation*}
  and since $(e_1^* - 2e_n^*)(e_1) = 1$ we get
  $\nnorm{e_1^* - 2e_n^*} = 1$. As a conclusion,
  the sequence $(e_1^* - 2e_n^*)_{n\geq 1}\subset S_{Y^*}$
  converges weakly to $e_1^*$ and satisfies 
  \begin{equation*}
    \nnorm{e_1^* - (e_1^* - 2e_n^*)} = 2\nnorm{e_n^*}
    = 2,
  \end{equation*}
  for every $n\geq 2$, so $e_1^*$ is a super $\Delta$-point
  in $B_{Y^*}$.
\end{proof}

We can actually say a bit more about how the points
$e_1$ and $e_1^*$ sit in their respective unit balls
-- they are both extreme points.

\begin{prop}\label{prop:e1_e1star_ext}
  We have that $e_1 \in \ext B_Y$ and $e_1 + e_n \in \ext B_Y$
  for all $n \ge 2$.
  Similarly, $e_1^* \in \ext B_{Y^*}$ and
  $e_1^* - 2e_n^* \in \ext B_{Y^*}$ for all $n \ge 2$.
\end{prop}

\begin{proof}
  Let $x := (x_n)_{n\geq1} \in \ell_2$.
  We have the following easy facts.
  \begin{enumerate}
  \item\label{item:e1_extreme_12}
    If $x_n>1$ for some $n \in \N$, then $x \notin B(\ell_2,\|\cdot\|)$.
  \item\label{item:e1_extreme_3}
    If $x_1=1$, and if $x_n<0$ for some $n\geq 2$,
    then $x\notin B_Y$.
  \end{enumerate}
  Let us see why \ref{item:e1_extreme_12} holds.
  We have
  \begin{equation*}
    B_{\ell_2} \cup \{\pm(e_1 +  e_n),\ n\geq 2\}
    \subset
    \{e_1^*\leq1\} \cap \bigcap_{n\geq 2} \{e_n^*\leq 1\},
  \end{equation*}
  and since this set is clearly (weakly) closed and convex,
  we also have
  \begin{equation*}
    B\subset \{e_1^*\leq1\}\cap \bigcap_{n\geq 2} \{e_n^*\leq 1\},
  \end{equation*}
  and \ref{item:e1_extreme_12} follows.
  \ref{item:e1_extreme_3} is clear by definition of $\nnorm{\cdot}$.
  From this it readily follows that $e_1$ and $e_1 + e_n$, $n \ge 2$,
  are extreme points in $B_Y$.

  Next, let $x^* := (x_n)_{n\geq1} \in \ell_2$.
  For the dual case we have following.
  \begin{enumerate}[resume]
  \item\label{item:e1star_extreme_12}
    If $x_1>1$ or $x_n > 2$ for some $n \in \N$, then $x^* \notin B_{Y^*}$.
  \item\label{item:e1star_extreme_3}
    If $x_1=1$, and if $x_n > 0$ for some $n\geq 2$,
    then $x^* \notin B_{Y^*}$.
  \end{enumerate}
  Similarly to the above, we have
  \begin{equation*}
    B_{(\ell_2,\|\cdot\|)^*} \cup \{\pm(e^*_1 -  2e^*_n),\ n\geq 2\}
    \subset
    \{e_1 \leq 1\} \cap \bigcap_{n\geq 2} \{e_n \leq 2\},
  \end{equation*}
  and since this set is clearly (weakly) closed and convex,
  we also have
  \begin{equation*}
    B_{Y^*} \subset \{e_1 \leq1\}\cap \bigcap_{n\geq 2} \{e_n \leq 2\},
  \end{equation*}
  and \ref{item:e1star_extreme_12} follows.
  \ref{item:e1star_extreme_3} is clear
  since $\nnorm{e_1 + e_n} = 1$ and $x^*(e_1 + e_n) = 1 + x_n$.

  From \ref{item:e1star_extreme_12} and \ref{item:e1star_extreme_3} it
  readily follows that
  $e_1^*$ and $e_1^* - 2e_n^*$ are extreme points in $B_{Y^*}$.
\end{proof}

Using notation from \cite{MPRZ} we can say more about
$e_1$ and $e_1^*$.
If $C$ is a convex set, then a given subset
$D$ of $C$ is a
\emph{convex combination of relatively weakly open subsets}
(\emph{ccw} for short) of $C$, if $D$ is of the form
\begin{equation*}
  D := \sum_{i=1}^n \lambda_i W_i,
\end{equation*}
where $n \in \N$, $\lambda_i \in (0,1]$ for $i = 1,\ldots,n$,
$\sum_{i=1}^n \lambda_i = 1$ and $W_i$ are relatively weakly open
subsets of $C$.

Then we have that $e_1$ and $e_1^*$ are not only super $\Delta$-points,
but actually ccw $\Delta$-points in the sense of \cite{MPRZ}.

\begin{cor}\label{cor:e1_e1s_ccw_delta}
  If $C$ is a ccw of $B_Y$ such that $e_1 \in C$,
  then $\sup_{y \in C} \nnorm{e_1 - y} = 2$.
  Similarly,
  if $D$ is a ccw of $B_{Y^*}$ such that $e_1^* \in D$,
  then $\sup_{y^* \in D} \nnorm{e_1^* - y^*} = 2$.
\end{cor}

\begin{proof}
  That super $\Delta$-points that are extreme
  are ccw $\Delta$-points is proved in \cite[Proposition~3.13]{MPRZ}
  (see also the remark following the proposition).
\end{proof}

To end the section, let us emphasize that the points $e_1$ and $e_1^*$ satisfy the strongest possible $\Delta$-property in $Y$ and $Y^*$ as they are both ccw $\Delta$, but that they fail to be Daugavet-points in an extreme way as they both belong to the closure of the set of all strongly exposed points in their respective unit balls. In particular, this example provides a positive answer to \cite[Question~7.12]{MPRZ}.

\section{Renorming any Banach space to have a
  \texorpdfstring{$\Delta$}{Delta}-point}
\label{sec:renorm-Banach-spaces}

From \cite[Corollary~6.10]{ALMP} we know that
no finite dimensional Banach space admits a $\Delta$-point.
The following theorem, which is the main theorem in this section,
highlights that the question of whether or not a Banach space
contains a $\Delta$-point is very much an isometric and
not an isomorphic question.
The proof combines ideas and results from the previous two sections.

\begin{thm}\label{thm:all_inf_dim_renorm_delta}
  Let $X$ be an infinite dimensional Banach space.
  Then the following holds:
  \begin{enumerate}
  \item\label{item:renorm1}
    If $X$ fails the Schur property, and in particular
    if $X$ does not contain a copy of $\ell_1$,
    then there exists an equivalent norm on $X$ for which
    $X$ contains a super $\Delta$-point;
  \item\label{item:renorm2}
    If $X$ contains a copy of $\ell_1$,
    then there exists an equivalent norm on $X$ for which
    $X$ contains a $\Delta$-point.
  \end{enumerate}
\end{thm}

As we will show in Section~\ref{sec:DualityForDelta}
(see Corollary~\ref{cor:renorming_Delta+weak*dualDelta}),
the existence of a $\Delta$-point in a Banach space automatically
implies the existence of a weak$^*$ super $\Delta$-point in its dual,
so we will essentially focus here on the construction of
$\Delta$-points in our target spaces.

Our first step is showing that we can use similar ideas to the ones 
from Section~\ref{sec:superr-banach-space} to renorm
any infinite dimensional Banach space that fails the Schur property 
with a super $\Delta$-point.

\begin{proof}[Proof of
  Theorem~\ref{thm:all_inf_dim_renorm_delta}~\ref{item:renorm1}.]
  That spaces which do not contain a copy of $\ell_1$ fail the Schur property
  follows from Rosenthal's $\ell_1$~theorem. Now if $X$ fails the Schur
  property, then using classic results on extraction of basic sequences
  (see e.g. \cite[Proposition~1.5.4]{AlbiacKalton}),
  we can construct a weakly null basic sequence $(e_n)_{n\geq 1}\subset S_X$.
  Let $(e_n^*)_{n\geq 1}$ be the sequence of biorthogonal functionals
  on the space $\csp{}(e_n)_{n \geq1}$, and for every $n\geq 2$, take
  a Hahn--Banach extension $f_n\in X^*$ of the functional $e_1^*-2e_n^*$.
  As $(e_n)_{n\geq 1}$ is basic, there exists a constant $K\geq 1$ such 
  that $\sup_{n\geq 2}\norm{f_n}\leq K$.
  We define an equivalent norm $\nnorm{\cdot}$ on $X$ by
  \begin{equation*}
    \nnorm{x}:=\max\left\{\frac{1}{2}\norm{x},\ \sup_{n\geq 2}\abs{f_n(x)}\right\}
  \end{equation*}
  for every $x\in X$.
  Then $\frac{1}{2}\norm{\cdot}\leq \nnorm{\cdot}\leq K\norm{\cdot}$,
  and by construction, we have, for every $n\geq 2$,
  \begin{equation*}
    \nnorm{e_1}=\nnorm{e_1+e_n}=1\ \text{and}\ \nnorm{e_n}=2.
  \end{equation*}
  As $e_1+e_n\to e_1$ weakly, it clearly follows that $e_1$
  is a super $\Delta$-point in $(X,\nnorm{\cdot})$.
\end{proof}

Before proving
Theorem~\ref{thm:all_inf_dim_renorm_delta}~\ref{item:renorm2}
let us state the following lemma that is an easy consequence
of a classical norm extension result.

\begin{lem}\label{lem:renorm_subsp_with_delta}
  Let $X$ be a Banach space.
  If there exists a subspace $Y$ of $X$ that can
  be renormed to admit a $\Delta$-point,
  then $X$ can be renormed to admit a $\Delta$-point.

  The same holds for super $\Delta$-points.
\end{lem}

\begin{proof}
  Assume that there is a norm $\abs{\cdot}$ on $Y$
  that admits a $\Delta$- or a super $\Delta$-point.
  By \cite[Lemma II.8.1]{DGZ} we can extend
  $\abs{\cdot}$ to an equivalent norm $\nnorm{\cdot}$ on $X$
  which coincides with $\abs{\cdot}$ on $Y$.
  In other words, $(Y,\abs{\cdot})$ is
  isometrically isomorphic to a subspace of $(X,\nnorm{\cdot})$,
  and since both $\Delta$- and super $\Delta$-points pass to superspaces,
  the conclusion follows.
\end{proof}

\begin{proof}[Proof of
  Theorem~\ref{thm:all_inf_dim_renorm_delta}~\ref{item:renorm2}.]
  If $X$ contains a subspace $Y$ which is isomorphic to $\ell_1$,
  then $Y$ can be renormed to admit a $\Delta$-point
  by Theorem~\ref{thm:veeorg_space}.
  We finish by using Lemma~\ref{lem:renorm_subsp_with_delta}.
\end{proof}

Let us end the present section with a few remarks.
  
\begin{rem}\label{rem:renorming_rmks}
  {\ }% Avoid enumerate starting on the first line
  \begin{enumerate}[label=(\alph*)]
  \item\label{item:renorm_rem1}
    Note that it is essential in the proof of 
    Theorem~\ref{thm:all_inf_dim_renorm_delta}~\ref{item:renorm1}
    to go first trough this process of extraction of a basic sequence in
    order to have complete control over the values of $f_n(e_m)$ for 
    distinct $m,n$. It is thus unclear whether a
    similar construction could be done on a weakly null normalized net,
    and in particular whether it could be implemented in $\ell_1$. 
    Also, it is still unknown whether the Daugavet point
    in the space $\lipfree{\mathcal{M}}$ from \cite[Example~3.1]{VeeorgStudia}
    studied in Section~\ref{sec:separable-dual-space} is a super
    $\Delta$-point (see Question~7.1 in \cite{MPRZ} for further discussions). So we do not know 
    whether $\ell_1$ can be renormed with a super $\Delta$-point. 
  \item\label{item:renorm_rem2}
    If $X$ is a Banach space with a normalized weakly null Schauder
    basis $(e_n)_{n\geq 1}$, then the construction from
    Section~\ref{sec:superr-banach-space} can be implemented in a
    natural way on this sequence in order to provide a
    renorming of $X$ for which $e_1$ is a super $\Delta$-point and
    $e_1^*$ is a weak$^*$ super $\Delta$-point.
    So using Lemma~\ref{lem:renorm_subsp_with_delta},
    we get an alternative geometric proof for
    Theorem~\ref{thm:all_inf_dim_renorm_delta}~\ref{item:renorm1}.
  \item\label{item:renorm_rem3}
    Let $X$ be a Banach space with a normalized weakly null Schauder
    basis $(e_n)_{n\geq 1}$. Up to renorming, we may assume that 
    $(e_n)_{n\geq 1}$ is bimonotone. Then it is straightforward to check
    that for either the renorming from 
    Theorem~\ref{thm:all_inf_dim_renorm_delta}~\ref{item:renorm1} or 
    for the renorming copied from Section~\ref{sec:superr-banach-space} that is
    discussed in item \ref{item:renorm_rem2} above, the point $e_1$
    is also an extreme point of the new ball, hence a ccw $\Delta$-point
    in the sense of \cite{MPRZ} (see Proposition~\ref{prop:e1_e1star_ext} and
    Corollary~\ref{cor:e1_e1s_ccw_delta}). 
    The same goes for $e_1^*$ if the
    basis is moreover assumed to be shrinking, and in this case 
    $e_1^*$ becomes a weak$^*$ ccw $\Delta$-point.
    However, it is unclear whether those points pass
    to superspaces in general, and thus we do not know whether
    Theorem~\ref{thm:all_inf_dim_renorm_delta}~\ref{item:renorm1}
    admits an analogue for ccw $\Delta$-points.
    
  \end{enumerate}
 \end{rem}

\section{Duality for \texorpdfstring{$\Delta$}{Delta}-points and
  applications}
\label{sec:DualityForDelta}

In this section, we provide a new powerful duality result for
$\Delta$-points, and collect a few striking consequences for the
geometry of Banach spaces.
The applications range from asymptotic geometry and unconditional bases
to Hahn--Banach smooth spaces.
The main theorem in this section is
Theorem~\ref{thm:D-point/weak*-Delta-points_implies_weak*-super-Delta-point}
where we prove that if a Banach space $X$ contains
a $\mathfrak{D}$-point, or if its dual $X^*$ contains
a weak$^*$ $\Delta$-point, then in both cases $X^*$ actually contains
a weak$^*$ super $\Delta$-point.
As a consequence, we show that $\mathfrak{D}$-points and
weak$^*$ $\Delta$-points are incompatible with some geometric
properties of Banach spaces, such as asymptotic smoothness, shrinking
or monotone boundedly complete $k$-unconditional bases with $k<2$, or
Hahn--Banach smooth spaces that have a dual space with the
Kadets--Klee property.

\subsection{Duality for \texorpdfstring{$\Delta$}{Delta}-points and
  asymptotic geometry}
\label{subsec:Delta-duality_and_asymptotic-geometry}

It was proved in \cite{ALMP} that $\Delta$-points are incompatible
with some asymptotic properties of smoothness and convexity of norms.
The following can be obtained by combining results
from \cite{ALMP} and \cite{VeeorgFunc}.

\begin{thm}\label{thm:asymptotic-properties_Delta-points}
  Let $X$ be a Banach space.
  If $X$ is asymptotically uniformly smooth, then $X$ contains no
  $\Delta$-point, and $X^*$ contains no weak$^*$ $\Delta$-point.
\end{thm}

The ideas behind this result were related to the duality between
asymptotic smoothness and weak$^*$ asymptotic convexity of the dual on
one side, and to considerations on the
Kuratowski measure of non-compactness $\alpha $ of weak$^*$
slices on the other.
More precisely the two following facts were obtained:
\begin{enumerate}
\item
  if a point $x\in S_X$ is a $\Delta$-point, then
  $\alpha\left( S(x, \delta) \right) = 2$
  for every $\delta > 0$
  \cite[Theorem~3.5]{ALMP} or \cite[Corollary~2.2]{VeeorgFunc}.
  In particular, no \emph{asymptotically smooth point} can be
  a $\Delta$-point \cite[Proposition~3.6]{ALMP}.
  That is, $x$ is not a $\Delta$-point
  if $\lim_{t \to 0} \bar{\rho}(t,x)/t = 0$, where
  $\bar{\rho}(t,x)$ is the modulus of asymptotic smoothness at $x$
  (see e.g. \cite[Definition~14.6.1]{AlbiacKalton}).
\item
  if a point $x^*\in S_{X^*}$ is a weak$^*$ $\Delta$-point,
  then every weak$^*$ slice $S$ of $B_{X^*}$ containing $x^*$ has
  Kuratowski measure $\alpha(S) = 2$
  \cite[Corollary~2.4]{VeeorgFunc}.
  In particular, no \emph{weak$^*$ quasi denting-point} can be
  a weak$^*$ $\Delta$-point.
\end{enumerate}

As every point in the unit sphere of a weak$^*$ asymptotically uniformly
convex dual space is weak$^*$ quasi denting
(see e.g. the discussion following Corollary~4.7 in \cite{ALMP}),
Theorem~\ref{thm:asymptotic-properties_Delta-points} immediately
follows.

The relation between asymptotic properties and $\mathfrak{D}$-points was
left aside in those papers, but let us point out that from the proof of \cite[Theorem~4.2]{ALMP}
we can collect the following lemma. Recall that a unit sphere element
$x$ in a Banach space $X$ is an $\alpha$-strongly exposed point if
there exists $x^* \in D(x)$ such that
$\lim_{\delta \to 0} \alpha(S(x^*, \delta)) = 0$.

\begin{lem}\label{lem:alpha_str_exp_not_Dpt}
  Let $X$ be a Banach space and let $x \in S_X$
  be an $\alpha$-strongly exposed point.
  Then $x$ is not a $\mathfrak{D}$-point.
\end{lem}

Recall that a Banach space $X$ has \emph{Rolewicz' property $(\alpha)$}
if for every $x^* \in S_{X^*}$ and $\varepsilon > 0$
there exists $\delta > 0$ such that
$\alpha(S(x^*,\delta)) \le \varepsilon$.
We say that $X$ has \emph{uniform property $(\alpha)$}
if the same $\delta$ works for all $x^* \in S_{X^*}$.
These properties were introduced by Rolewicz in \cite{MR928575}.
Implicit in Rolewicz \cite[Theorem~3]{MR928575}
is the result that $X$ is asymptotically uniformly convex and
reflexive if and only if $X$ has uniform property $(\alpha)$;
Rolewicz uses the term ``$X$ is $\Delta$-uniformly
convex'' instead of $X$ is asymptotically uniformly convex and
reflexive.

As corollaries of Lemma~\ref{lem:alpha_str_exp_not_Dpt}, we get the two following results.
The first corollary is a strengthening of
\cite[Theorem~4.4]{ALMP}.

\begin{cor}\label{cor:property-alpha_D-points}
  Let $X$ be a Banach space.
  If $X$ has Rolewicz' property $(\alpha)$, and in particular if $X$
  has finite dimension or if $X$ is reflexive and asymptotically uniformly
  convex, then $X$ contains no $\mathfrak{D}$-point.
\end{cor}

\begin{proof}
  This immediately follows from  Lemma~\ref{lem:alpha_str_exp_not_Dpt}
  since from Rolewicz' property $(\alpha)$,
  every $x \in S_X$ is $\alpha$-strongly exposed.
\end{proof}

\begin{cor}\label{cor:dual_aucst_no_delta}
  Let $X$ be a Banach space such that
  $X^*$ is weak$^*$ asymptotically uniformly convex.
  Then $X^*$ contains no weak$^*$ $\Delta$-points and
  no $x^* \in S_{X^*}$ that attains its norm on $X$
  is a $\mathfrak{D}$-point.
\end{cor}

\begin{proof}
  That $X^*$ has no weak$^*$ $\Delta$-points
  is part of Theorem~\ref{thm:asymptotic-properties_Delta-points}.
  Let $x^* \in S_{X^*}$ be such that there exists
  $x \in S_X$ with $x^*(x) = 1$.
  Then $x^*$ is (weak$^*$) $\alpha$-strongly exposed
  by \cite[Corollary~3.4]{ALMP}, hence not a $\mathfrak{D}$-point
  by Lemma~\ref{lem:alpha_str_exp_not_Dpt}.
\end{proof}

\begin{rem}\label{rem:dual_of_c_0_AUC*_D-points}
  Note that this result is sharp, because as the dual of $c_0$,
  we have that $\ell_1$ is weak$^*$ asymptotically uniformly convex
  and that every $x \in S_{\ell_1}$ with infinite
  support is a $\mathfrak{D}$-point
  (see Proposition~2.3 in \cite{AHLP}).
\end{rem}

It is unclear whether either \cite[Theorem~3.5]{ALMP} or
\cite[Corollary~2.2]{VeeorgFunc} admit analogues for
$\mathfrak{D}$-points.  Yet we can provide
stronger duality results for those points that will
in particular give new information about asymptotically
smooth points and asymptotically  uniformly smooth spaces.

\begin{thm}\label{thm:D-point/weak*-Delta-points_implies_weak*-super-Delta-point}
  Let $X$ be a Banach space.
  If $X$ contains a $\mathfrak{D}$-point,
  or if $X^*$ contains a weak$^*$ $\Delta$-point,
  then $X^*$ contains a weak$^*$ super $\Delta$-point.
\end{thm}

\begin{proof}
  First, let us assume that $x \in S_X$ is a $\mathfrak{D}$-point.
  We will distinguish between two cases.
  First assume that the set $D(x)$ contains exactly one element $x^*$.
  For every $n\in\N$ there exists $x_n \in S(x^*, 1/n)$
  such that $\|x - x_n\| > 2 - 1/n$.
  Let $x_n^* \in S_{X^*}$ be such that
  $x_n^*(x) - x_n^*(x_n) > 2 - 1/n$.
  By weak$^*$ compactness of $B_{X^*}$, there exists a subnet
  $(y_\alpha^*)_{\alpha\in\mathcal{A}}$ of $(x_n^*)$ that is
  weak$^*$-convergent to some element $y^* \in B_{X^*}$.
  Since $x_n^*(x) > 1 - \frac{1}{n}$ for every $n\in \N$,
  we get $x_n^*(x) \rightarrow 1$ and thus also
  $y_\alpha^*(x) \rightarrow 1$. Hence $y^*\in D(x)$,
  i.e. $y^* = x^*$.
  Furthermore,
  \begin{equation*}
    \|x^* - x_n^*\|
    \ge x^*(x_n) - x_n^*(x_n)
    > 2 - \frac{2}{n}
  \end{equation*}
  for every $n\in\N$.
  Therefore $\|x^* - x_n^*\| \rightarrow 2$ and thus
  also $\|x^* - y_\alpha^*\| \rightarrow 2$,
  meaning $x^*$ is a weak$^*$ super $\Delta$-point.

  Now assume that $D(x)$ has at least two distinct elements.
  Since $D(x)$ is convex, then $D(x)$ is infinite.
  Denote by $\mathcal{A}$ the directed set of all finite subsets of
  $D(x)$ ordered by inclusion.

  For every $A\in \mathcal{A}$ we have
  $\frac{1}{|A|} \sum_{x^*\in A} x^*(x) = 1$,
  thus there exists $x_A\in S_X$ such that
  \begin{equation*}
    \frac{1}{|A|} \sum_{x^*\in A} x^*(x_A)
    >
    1 - \frac{1}{|A|^2}
  \end{equation*}
  and $\|x - x_A\| > 2 - \frac{1}{|A|}$.
  Therefore there also exists $x_A^* \in S_{X^*}$
  such that $x_A^*(x - x_A) > 2 - \frac{1}{|A|}$.
  There exists a subnet $(y_B^*)_{B\in\mathcal{B}}$ of
  $(x_A^*)_{A\in\mathcal{A}}$ that is
  weak$^*$-convergent to some element $y^* \in B_{X^*}$.
  Since $x_A^*(x) > 1 - \frac{1}{|A|}$ for every $A\in \mathcal{A}$,
  we get $x_A^*(x)\rightarrow 1$ and thus also $y_B^*(x)\rightarrow
  1$.
  Hence $y^*\in D(x)$.

  Fix $\alpha > 0$.
  Let $A_0 \in \mathcal{A}$ be such that $y^* \in A_0$ and
  $\frac{2}{|A_0|} < \alpha$.
  Then for all $A \succeq A_0$ we have $y^*\in A$ and thus
  \begin{equation*}
    \frac{1}{|A|} y^*(x_A)
    \ge
    \frac{1}{|A|}\sum_{x^*\in A} x^*(x_A) - 1 + \frac{1}{|A|}
    >
    \frac{1}{|A|} - \frac{1}{|A|^2},
  \end{equation*}
  which gives us
  \begin{equation*}
    \|y^* - x_A^*\|
    \ge
    y^*(x_A) - x_A^*(x_A)
    >
    1 - \frac{1}{|A|} + 1 - \frac{1}{|A|}
    \ge
    2 - \frac{2}{|A_0|}
    >
    2 - \alpha.
  \end{equation*}
  Therefore $\|y^*-x_A^*\|\rightarrow 2$ and thus also
  $\|y^*-y_B^*\|\rightarrow 2$, meaning
  $y^*$ is a weak$^*$ super $\Delta$-point.

  Second, let us assume that $x_0^*\in S_{X^*}$ is a weak$^*$ $\Delta$-point.
  Let $x_0\in S_{X}$ be such that $x_0^*(x_0) > 1/2$.
  We construct recursively two sequences $(x_i)_{i\geq 1}$
  and $(x_i^*)_{i\geq 1}$ in $S_X$ and $S_{X^*}$ in the following way.
  First choose $x_1^*\in S(x_0,1/2)$ such that
  $\|x_0^* - x_1^*\| > 2 - 1/4$,
  and choose $x_1\in S_{X}$ such that
  $x_0^*(x_1) > 1 - 1/4$ and $x_1^*(-x_1) > 1 - 1/4$.

  Then assume that we have found $x_1,\ldots,x_{n-1} \in S_{X}$
  and $x_1^*, \ldots, x_{n-1}^* \in S_{X^*}$ such that
  \begin{equation*}
    -x_k^*(x_k) > 1 - \frac{1}{2^{k+1}}
    \quad\text{ and }\quad
    \sum_{i=0}^{k-1} x_k^*(x_i) > k - 1
    \quad\text{ and }\quad
    x_0^*(x_k) > 1 - \frac{1}{2^{k+1}}
  \end{equation*}
  for every $k \in \{1,\ldots,n-1\}$.
  Then
  \begin{equation*}
    x_0^* \Big(\sum_{i=0}^{n-1}x_i\Big)
    >
    \sum_{i=0}^{n-1} \Big(1 - \frac{1}{2^{i+1}}\Big)
    >
    n - 1.
  \end{equation*}
  Since $x_0^*$ is a weak$^*$ $\Delta$-point,
  there exist $x_n^*\in S_{X^*}$ such that
  \begin{equation*}
    x_n^*\Big(\sum_{i=0}^{n-1}x_i\Big) > n - 1
  \end{equation*}
  and $\|x_0^* - x_n^*\| > 2 - 1/2^{n+1}$.
  Choose $x_n\in S_{X}$ such that
  $x_0^*(x_n) > 1-1/2^{n+1}$ and $-x_n^*(x_n) > 1-1/2^{n+1}$.

  So we end up with two sequences $(x_i)_{i\geq 1}$ and
  $(x_i^*)_{i\geq 1}$ in $S_X$ and $S_{X^*}$ such that
  $x_i^*(x_i) \to -1$.
  Furthermore, we have
  \begin{equation*}
    \sum_{i\in I} x_n^*(x_i)
    \ge
    \sum_{i=0}^{n-1} x_n^*(x_i) - (n - |I|)
    >
    n - 1 - (n - |I|)
    =
    |I| - 1
  \end{equation*}
  for every finite set $I\subseteq\{1,\dots,n-1\}$.
  As $B_{X^*}$ is weak$^*$ compact, there exist subnets
  $(y_\alpha)_{\alpha\in\mathcal{A}}$ and
  $(y_\alpha^*)_{\alpha\in\mathcal{A}}$
  of those sequences such that
  $(y_\alpha^*)_{\alpha\in\mathcal{A}}$ is weak$^*$-convergent
  to some element $y^* \in B_{X^*}$.
  Since
  \begin{equation*}
    x_n^*\Big(\frac{1}{|I|} \sum_{i\in I}x_i\Big)
    >
    1 - \frac{1}{|I|}
  \end{equation*}
  whenever $n>\max I$, we get
  \begin{equation*}
    y^*\Big(\frac{1}{|I|} \sum_{i\in I}x_i\Big)
    >
    1 - \frac{1}{|I|}
  \end{equation*}
  for every finite set $I\subseteq\N$.
  We will finally show that there exists a subnet
  $(z_\beta)_{\beta\in\mathcal{B}}$
  such that $y^*(z_\beta)\rightarrow1$.
  Fix $\gamma > 0$ and $\alpha_0 \in \mathcal{A}$.
  Then we can find a finite set $A \subseteq\{\alpha\in \mathcal{A}\colon \alpha\succeq \alpha_0\}$
  such that $1/|A| < \gamma$ and all elements in
  $A$ correspond to different elements in $\N$.
  Then
  \begin{equation*}
    y^*
    \Big(
    \frac{1}{|A|}\sum_{\alpha\in A} y_\alpha
    \Big)
    >
    1 - \frac{1}{|A|}
  \end{equation*}
  and thus there exists $\alpha\in A$
  such that $y^*(y_\alpha) > 1 - 1/|A| > 1 - \gamma$.
  It follows that $1$ is a cluster point of the net
  $(y^*(y_\alpha))_{\alpha\in\mathcal{A}}$,
  and therefore there exist subnets $(z_\beta)_{\beta\in\mathcal{B}}$
  and $(z_\beta^*)_{\beta\in\mathcal{B}}$ such that
  $y^*(z_\beta)\rightarrow1$.
  As we are working with subnets,
  $z_\beta^*(z_\beta) \rightarrow -1$
  by construction of the original sequences, so
  \begin{equation*}
    \|y^* - z_\beta^*\|
    \ge
    y^*(z_\beta) - z_\beta^*(z_\beta)
    \rightarrow 2.
  \end{equation*}
  Therefore $y^*$ is a weak$^*$ super $\Delta$-point.
\end{proof}

\begin{rem}\label{rem:weak*-super-Delta}
  It is known that there are $\Delta$-points that are not
  super $\Delta$, and that there even exists a Banach space with
  a large subset of Daugavet-points that contains no
  super $\Delta$-point
  (combine Proposition~2.12 and Theorem~3.1 or Theorem~4.7 in
  \cite{ALMT}).
  So the previous result is completely specific to the
  weak$^*$ topology, and it is quite clear from the above proof that
  essentially all comes down to the weak$^*$ compactness
  of the dual unit ball.
\end{rem}

Note that as a corollary of this result and
Theorem~\ref{thm:all_inf_dim_renorm_delta},
we get the following.

\begin{cor}\label{cor:renorming_Delta+weak*dualDelta}

    Let $X$ be an infinite dimensional Banach space. Then $X$ can be renormed so that $X$ admits a $\Delta$-point and $X^*$ admits a weak$^*$ super $\Delta$-point. Moreover, if $X$ fails the Schur property, then $X$ can be renormed so that $X$ admits a super $\Delta$-point and $X^*$ admits a weak$^*$ super $\Delta$-point. 
    
\end{cor}

Recall that a (dual) Banach space $X$ has the
(weak$^*$) \emph{Kadets property}
if the weak (respectively the weak$^*$)
and norm topology coincide on the unit sphere of $X$.
As pointed out in \cite[Section~3]{MPRZ}, it is clear that
(weak$^*$) super $\Delta$-points are incompatible with
being (weak$^*$) Kadets.
So observe that it immediately follows from
Theorem~\ref{thm:D-point/weak*-Delta-points_implies_weak*-super-Delta-point}
that if a Banach space $X$ contains a $\mathfrak{D}$-point,
or if its dual $X^*$ contains a weak$^*$-$\Delta$-point,
then $X^*$ fails to be weak$^*$ Kadets.
In particular, as weak$^*$ asymptotically uniformly convex duals satisfy a
uniform weak$^*$ Kadets property in the sense of \cite{Lancien95},
we immediately get the following improved version of
Theorem~\ref{thm:asymptotic-properties_Delta-points}.

\begin{cor}\label{cor:sharp_asymptotic-properties_Delta-points}
  Let $X$ be a Banach space.
  If $X$ is asymptotically uniformly smooth,
  then $X$ fails to contain $\mathfrak{D}$-points,
  and $X^*$ fails to contain weak$^*$ $\Delta$-points.
\end{cor}

Also observe that we can clearly get something
more precise out of the proof of
Theorem~\ref{thm:D-point/weak*-Delta-points_implies_weak*-super-Delta-point}.
Indeed, looking back there, we can see that if a point
$x\in S_X$ is a $\mathfrak{D}$-point,
then the set $D(x)$ contains a weak$^*$ super $\Delta$-point.
So as a corollary, we get the following pointwise result
(see also the discussion following
Theorem~\ref{thm:asymptotic-properties_Delta-points}).

\begin{cor}\label{cor:asymptotically-smooth-points_not_D-points}
  No asymptotically smooth point is a $\mathfrak{D}$-point.
\end{cor}

\begin{proof}
  Let $X$ be a Banach space.
  If $x\in S_X$ is a $\mathfrak{D}$ point,
  then as observed in the previous discussion,
  the set $D(x)$ contains a weak$^*$ super $\Delta$-point $y^*$.
  As this point is also a weak$^*$ $\Delta$-point,
  it follows from \cite[Corollary~2.4]{VeeorgFunc} that $\alpha(S)=2$
  for every weak$^*$ slice $S$ of $B_{X^*}$ containing $y^*$.
  In particular, $\alpha(S(x, \delta)) = 2$ for every
  $\delta > 0$ since $y^*(x) = 1$.
  Then by \cite[Corollary~3.4]{ALMP},
  $x$ is not an asymptotically smooth point
  (as the Kuratowski measure of the weak$^*$ slices of the dual unit ball
  defined by an asymptotically smooth point goes to $0$ as
  $\delta$ goes to $0$).
\end{proof}

Finally, note that we also get the following $\mathfrak{D}$-version
of \cite[Corollary~4.6]{ALMP}.

\begin{cor}\label{cor:auc_refl_no_Dpts}
  If $X$ is asymptotically uniformly convex and reflexive,
  then neither $X$ nor $X^*$ contain $\mathfrak{D}$-points.
\end{cor}

\begin{proof}
  Since $X$ is reflexive and asymptotically uniformly convex,
  $X$ has no $\mathfrak{D}$-points by
  Corollary~\ref{cor:property-alpha_D-points}.
  Since $X^*$ is asymptotically uniformly smooth
  $X^*$ has no $\mathfrak{D}$-points by
  Corollary~\ref{cor:sharp_asymptotic-properties_Delta-points}.
\end{proof}

\subsection{Shrinking and boundedly complete unconditional bases}
\label{subsec:unconditional_bases}

We know from \cite{ALMT} that there exists
a Banach space $X$ with a $1$-unconditional basis
that admits a Daugavet point. This $X$ contains
copies of both $c_0$ and $\ell_1$.
In this section we show that this is no coincidence,
a consequence of our results is that if $X$ has
a shrinking or boundedly complete $1$-unconditional basis then
$X$ cannot contain $\Delta$-points
(see Corollary~\ref{cor:uncond_basis_c0l1} below
for the precise statement).

Instead of working with a basis we can work
with an approximating sequence of compact operators.
A kind of reverse bounded approximation property
will then prevent the existence of weak$^*$ super $\Delta$-points
in the dual in a similar way to \cite[Proposition~3.20]{MPRZ},
and as a consequence of
Theorem~\ref{thm:D-point/weak*-Delta-points_implies_weak*-super-Delta-point},
it will also prevent the existence of $\mathfrak{D}$-point in the space
and of weak$^*$ $\Delta$-points in the dual.
So here is the main theorem of this subsection.

\begin{thm}\label{thm:shrinking_RMAP_no_delta}
  Let $X$ be a Banach space.
  Assume that there exists a family
  $(T_\lambda)_{\lambda \in \Lambda}$ of compact operators on $X$
  with $\sup_{\lambda \in \Lambda} \norm{I_X - T_\lambda} < 2$,
  such that $T_\lambda^* \to I_{X^*}$ in the strong
  operator topology.
  Then $X$ contains no $\mathfrak{D}$-points,
  and $X^*$ contains no weak$^*$ $\Delta$-points.
\end{thm}

\begin{proof}[Proof of Theorem~\ref{thm:shrinking_RMAP_no_delta}]
  Let $\eps>0$ be such that
  \begin{equation*}
    \sup_{\lambda \in \Lambda} \norm{I - T_\lambda} + 4\eps \le 2.
  \end{equation*}

  By
  Theorem~\ref{thm:D-point/weak*-Delta-points_implies_weak*-super-Delta-point},
  it is enough to show that if a net
  $(x_\alpha^*)_{\alpha\in\mathcal{A}} \subset B_{X^*}$ converges
  weak$^*$ to some $y^* \in B_{X^*}$,
  then for every $\alpha_0 \in \mathcal{A}$,
  there exists $\alpha \succeq \alpha_0$ such that
  \begin{equation*}
    \|y^* - x_{\alpha}^*\| \le 2 - 2\eps.
  \end{equation*}

  Pick $\lambda \in \Lambda$ such that
  $\norm{y^* - T_\lambda^*y^*} \leq \eps$.
  By Schauder's theorem, $T_\lambda^*$ is continuous
  as a map from $(B_{X^*},w^*)$ to $(X^*,\|\cdot\|)$
  (see e.g. \cite[Theorem~6.26]{BowersKalton})
  and it follows that
  $(T_\lambda^*(x_\alpha^*))_{\alpha\in\mathcal{A}}$
  converges in norm to $T_\lambda^*(y^*)$.
  Thus there exists $\alpha_0 \in\mathcal{A}$ such that
  $\norm{T_\lambda^*(y^* - x_{\alpha}^*)} \leq  \eps$
  for every $\alpha\succeq \alpha_0$.

  Then for every $\alpha\succeq \alpha_0$, we get
  \begin{align*}
    \norm{y^* - x_\alpha^*}
    &\leq
    \norm{y^* - T_\lambda^*y^*}
    +
    \norm{T_\lambda^*(y^* - x_\alpha^*)}
    +
    \norm{x_\alpha^* - T_\lambda^*x_\alpha^*}
    \\
    &\leq
    \norm{I^* - T_\lambda^*} + 2\eps
    \\
    &\leq
    2 - 2\eps,
  \end{align*}
  and the conclusion follows.
\end{proof}

\begin{rem}
  In Theorem~\ref{thm:shrinking_RMAP_no_delta} we assume that
  $X^*$ satisfies a version of the bounded compact approximation
  property with conjugate operators. It is clear from
  the above proof that actually all we need is a family
  $\mathcal{A} \subset \mathcal{K}(X)$ such that
  $\sup_{T \in \mathcal{A}} \|I_X - T\| < 2$ and such that
  given $\varepsilon > 0$ and $x^* \in X^*$ there exists
  $T \in \mathcal{A}$ such that $\|x^* - T^* x^*\| < \varepsilon$.
\end{rem}

Let us now collect a few corollaries from
Theorem~\ref{thm:shrinking_RMAP_no_delta}.

\begin{cor}\label{cor:refl_ufdd_no_Dpts}
  Let $X$ be a reflexive Banach space with a shrinking basis
  with partial sum projections $(P_k)$.
  If $\sup_{k \in \N} \|I - P_k\| < 2$,
  then $X$ and $X^*$ contain no $\mathfrak{D}$-points.

  In particular, if $X$ is a reflexive Banach space with
  $k$-unconditional basis for $k<2$,
  then $X$ and $X^*$ contain no $\mathfrak{D}$-points.
\end{cor}

In \cite[Theorem~2.17]{ALMT} it was shown that Banach spaces
with a subsymmetric basis have no $\Delta$-points and this
can be applied to show that the Schlumprecht space has no
$\Delta$-points.
Unlike previous results in this direction
Corollary~\ref{cor:refl_ufdd_no_Dpts} applies also to
the Tsirelson space and many of its relatives.

% James - Bases in Banach spaces
From Theorem~\ref{thm:shrinking_RMAP_no_delta} and the classic results
from James on bases in Banach spaces, we also get the result announced
at start of this subsection.

\begin{cor}\label{cor:uncond_basis_c0l1}
  Let $X$ be a Banach space with
  shrinking $k$-unconditional basis for $k<2$,
  then $X$ contains no $\mathfrak{D}$-points
  and $X^*$ contains no weak$^*$ $\Delta$-points.

  Let $X$ be a Banach space with a monotone
  boundedly complete $k$-unconditional basis for $k<2$,
  then $X$ contains no $\Delta$-points.

  In particular, a Banach space with 1-unconditional basis
  and a $\Delta$-point contains a copy of $c_0$ and $\ell_1$.
\end{cor}

\begin{proof}
  The first part follows directly from
  Theorem~\ref{thm:shrinking_RMAP_no_delta}
  using the partial sum projections.

  If $X$ has a monotone boundedly complete basis $(e_n)$
  then the biorthogonal functionals $(e_n^*)$ is a shrinking
  basis for $Z = \csp{}(e_n^*)$ and $X$ is isometric to
  $Z^*$ (c.f. e.g \cite[Theorem~3.2.15]{AlbiacKalton}).
  The unconditionality constants for $(e_n)$ and $(e_n^*)$ are the
  same hence the second part follows from the first.

  Finally, if $(e_n)$ is an unconditional basis for $X$, then
  it fails to be shrinking if and only if $X$ contains
  a subspace isomorphic to $\ell_1$
  (c.f. e.g. \cite[Theorem~3.3.1]{AlbiacKalton})
  and
  it fails to be boundedly complete if and only if $X$
  contains a subspace isomorphic to $c_0$
  (c.f. e.g. \cite[Theorem~3.3.2]{AlbiacKalton}).
  Since a 1-unconditional basis is monotone the last part
  of the statement follows from the first two.
\end{proof}

\begin{rem}
  Again, note that those results are sharp, as
  every norm one element of $\ell_1$ with infinite support
  is a $\mathfrak{D}$-point
  (see \cite[Proposition~2.3]{AHLP})
  and every norm-one sequence in $c$ converging
  to 1 is a Daugavet point \cite[Theorem~3.4]{AHLP}.

  Also, note that if we view $\ell_1$ as the dual of $c$,
  then $\ell_1$ contains a weak$^*$ super $\Delta$-point
  by
  Theorem~\ref{thm:D-point/weak*-Delta-points_implies_weak*-super-Delta-point}.
  However, $\ell_1$ as the dual of $c_0$ contains no
  weak$^*$ $\Delta$-points by Corollary~\ref{cor:uncond_basis_c0l1}.
\end{rem}

We have one final corollary of this subsection.
Recall that $\mathfrak{D}$-points can be lifted from
subspaces.

\begin{cor}\label{cor:subsp_of_uncond_basis_no_D-pt}
  If $X$ is a subspace of a Banach space with
  a shrinking $k$-unconditional basis for $k<2$,
  then $X$ contains no $\mathfrak{D}$-points.
\end{cor}

Let us note that there is no quotient version of
Corollary~\ref{cor:subsp_of_uncond_basis_no_D-pt}
since every separable Banach space is a quotient of $\ell_1$.
However, it is natural to ask the following.

\begin{quest}
  Can we conclude that
  $X^*$ contains no weak$^*$ $\Delta$-points
  if $X$ is a subspace of a Banach space with
  a shrinking $k$-unconditional basis for $k<2$?
\end{quest}

According to Cowell and Kalton \cite{MR2646086}
a separable Banach space $X$ has \emph{property $(au^*)$}
if
$\lim_{n \to \infty} \|x^* + x_n^*\| =
\lim_{n \to \infty} \|x^* - x_n^*\|$
whenever $x^* \in X^*$, $(x_n^*)_{n \geq 1}$ is a weak$^*$-null
sequence and both limits exist.
In particular, spaces with Kalton's property (M$^*$)
and spaces with a 1-unconditional finite dimensional decomposition
satisfy this property.

It was shown in \cite[Theorem~4.2]{MR2646086} that
a separable Banach space $X$ has property $(au^*)$
if and only if for every $\eps > 0$ we have that
$X$ is $(1+\eps)$-isomorphic to a subspace of
a Banach space $Y$ with a shrinking 1-unconditional basis.
This is in turn equivalent to the fact that for every
$\eps > 0$ we have that $X$ is isometric to a subspace
of a Banach space $Y$ with a shrinking
$(1+\eps)$-unconditional basis.
Hence by Corollary~\ref{cor:subsp_of_uncond_basis_no_D-pt}
separable Banach spaces with property $(au^*)$
have no $\mathfrak{D}$-points.

\subsection{Sequential super points and Hahn--Banach smooth spaces}
\label{subsubsec:Hahn-Banach-smooth-spaces}

Among the first non-reflexive examples of Banach spaces
with no $\Delta$-points were some M-embedded spaces
and their duals, e.g. $c_0$ and $\ell_1$,
$\mathcal{K}(\ell_2)$ and its dual,
and the Schreier space and its dual.
For the Schreier space and its dual this was shown in
\cite{ALM}, but it also follows from
Corollary~\ref{cor:uncond_basis_c0l1}.
The first two are covered by
Corollary~\ref{cor:sharp_asymptotic-properties_Delta-points}
and, in fact, both $c_0$ and $\mathcal{K}(\ell_2)$
have Kalton's property (M$^*$).
A not completely unnatural question to ask is the following:

\begin{quest}\label{quest:M-embedded}
  Let $X$ be a non-reflexive M-embedded Banach space.
  Do $X$ and its dual $X^*$ fail to contain $\Delta$-points?
\end{quest}

In this section we will show that the answer is positive
when $X^*$ is a Lipschitz-free space or
more generally when $X^*$ is Kadets--Klee.
However, the following example shows that in general
the answer is negative.

\begin{example}\label{exmp:non_refl_M_with_Delta}
  There exists a non-reflexive M-embedded Banach space $X$
  such that both $X$ and $X^*$ have a super $\Delta$-point.

  Let $Y := (\ell_2,\nnorm{\cdot})$ be the renorming
  of $\ell_2$ from Theorem~\ref{thm:l2_renorm_e1_delta_not_daugavet}.
  Then both $Y$ and $Y^*$ admit a super $\Delta$-point.
  Let $X := c_0(Y)$. Then both $X$ and $X^* = \ell_1(Y^*)$
  isometrically contain a subspace with a super $\Delta$-point.
  Since such points can be lifted to any superspace
  both $X$ and $X^*$ admit a super $\Delta$-point.
  Reflexive spaces are trivially M-embedded hence
  $X$ is M-embedded by \cite[Theorem~III.1.6]{HWW}.
\end{example}

Before we move on to Hahn--Banach smooth spaces in more detail,
let us point out that
Corollary~\ref{cor:sharp_asymptotic-properties_Delta-points} already
provides a positive answer to Question~\ref{quest:M-embedded} for
$X^* = \lipfree{M}$ where $M$ is a proper purely 1-unrectifiable metric space.
Indeed, we have the following result.

\begin{cor}\label{cor:proper_dual_no_delta}
  If $M$ is a proper metric space,
  then $\lip_0(M)$ is M-embedded and does not admit
  $\mathfrak{D}$-points.

  If, in addition, $M$ is a proper purely 1-unrectifiable metric space,
  then $\lipfree{M}$, the dual of $\lip_0(M)$,
  has no weak$^*$ $\Delta$-points.
\end{cor}

\begin{proof}
  Let $M$ be a proper metric space.
  Dalet showed that, for any $\varepsilon > 0$,
  the space $\lip_0(M)$ is $(1+\varepsilon)$-isomorphic
  to a subspace of $c_0$ \cite[Lemma~3.9]{MR3376824}.
  From the three-ball property \cite[Theorem~I.2.2]{HWW}
  it follows that $\lip_0(M)$ is M-embedded.
  It also follows that $\lip_0(M)$ is asymptotically uniformly smooth
  (see e.g. \cite[Lemma~4.4.1]{CP}).
  By Corollary~\ref{cor:sharp_asymptotic-properties_Delta-points}
  we get that $\lip_0(M)$ does not admit $\mathfrak{D}$-points
  and its dual has no weak$^*$ $\Delta$-points.

  It is known that if $M$ is a proper metric space,
  then $\lipfree{M}$ is a dual space if and only if $M$ is
  purely 1-unrectifiable and, in that case,
  $\lip_0(M)$ is a predual \cite[Theorem~3.2]{AGPP}.
  Note that duals of M-embedded spaces are L-summands in their biduals
  and any Banach space has at most one predual that is M-embedded
  \cite[Proposition~IV.1.9]{HWW}.
\end{proof}

It was proved by Fabian and Godefroy \cite[Theorem~3]{FG} that every
M-embedded space is weakly compactly generated and Asplund.
In particular, the dual of any M-embedded space has the RNP.
In fact, every M-embedded space $X$ is
\emph{Hahn--Banach smooth} (see e.g.  \cite[Proposition~I.1.12]{HWW}),
meaning that every $x^*\in X^*$ has a unique norm-preserving extension
to $X^{**}$. 

Sullivan proved in \cite{Sullivan} that for a
Hahn--Banach smooth space $X$, the relative weak and weak$^*$
topologies on $B_{X^*}$ agree on $S_{X^*}$
(also see \cite[III.Lemma~2.14]{HWW}).
So a direct consequence of
Theorem~\ref{thm:D-point/weak*-Delta-points_implies_weak*-super-Delta-point}
is that if $X$ is Hahn--Banach smooth, and $X$ admits a
$\mathfrak{D}$-point or $X^*$ admits a weak$^*$ $\Delta$-point,
then $X^*$ admits a super $\Delta$-point.
In particular, $X^*$ fails the Kadets property.

Note that the sequential version of the Kadets property
(the Kadets--Klee property, see below for the definition) is not incompatible
with super $\Delta$-points in general since
there exists a Banach space with the Daugavet property
and the Schur property \cite{KW_Schur+Daugavet}.
However, it was pointed out in \cite[Remark~3.5]{MPRZ} that natural
sequential versions of those points would automatically prevent the
Schur property.
So let us introduce the following.

\begin{defn}
  Let $X$ be a Banach space.
  We say that
  \begin{enumerate}
  \item
    a point $x\in S_X$ is a \emph{sequential super $\Delta$-point}
    if there exists a sequence $(x_n)_{n\geq 1}$ in $S_X$
    that converges weakly to $x$ and such that $\norm{x-x_n}\to 2$.
  \item
    a point $x^*\in S_{X^*}$ is
    a \emph{weak$^*$ sequential super $\Delta$-point}
    if there exists a sequence $(x_n^*)_{n\geq 1}$ in $S_{X^*}$
    that converges weak$^*$ to $x^*$ and such that
    $\norm{x^*-x_n^*}\to 2$.
  \end{enumerate}
\end{defn}

Clearly, every (weak$^*$) sequential super $\Delta$-point is a
(weak$^*$) super $\Delta$-point,
and as we pointed out in the previous discussion,
the converse does not hold in general.

However, it is clear that weak$^*$ sequential super $\Delta$-points 
coincide with weak$^*$ super $\Delta$-points in duals of separable spaces.
With respect to the weak topology, recall that it is always possible by the results from
Rosenthal to extract from any given bounded weakly converging net 
in a space that does not contain $\ell_1$ a weakly converging 
sequence, so sequential super $\Delta$-point and super $\Delta$-points always coincide in this 
context. In particular the two notions are equivalent in Asplund spaces.

As mentioned, sequential super $\Delta$-points are incompatible with
the Schur property.
More precisely, they are incompatible with the Kadets--Klee property.
Recall that a (dual) Banach space $X$  is said to be (weak$^*$)
\emph{Kadets--Klee} if every for every sequence
$(x_n)_{n\geq 1}$ in $S_X$, and for every $x \in S_X$,
we have that $x_n\to x$ weakly (respectively weak$^*$)
if and only if $\norm{x-x_n}\to 0$.

Clearly, the existence of a (weak$^*$) sequential super $\Delta$-point
is incompatible with the (weak$^*$) Kadets--Klee property.

It turns out that
Theorem~\ref{thm:D-point/weak*-Delta-points_implies_weak*-super-Delta-point}
actually provides a weak$^*$ sequential super $\Delta$-point under the
assumption that the unit ball of the dual space is weak$^*$
sequentially compact.

\begin{thm}
  \label{thm:sequential_D-point/weak*-Delta-points_implies_weak*-super-Delta-point}
  Let $X$ be a Banach space such that
  $B_{X^*}$ is weak$^*$ sequentially compact.
  If $X$ contains a $\mathfrak{D}$-point, or if $X^*$ contains a
  weak$^*$ $\Delta$-point,
  then $X^*$ contains a sequential weak$^*$ super $\Delta$-point.
\end{thm}

\begin{proof}
  If $X$ contains a $\mathfrak{D}$-point,
  then by
  Theorem~\ref{thm:D-point/weak*-Delta-points_implies_weak*-super-Delta-point}
  $X^*$ admits a weak$^*$ super $\Delta$-point,
  hence a weak$^*$ $\Delta$-point,
  so it is sufficient to prove the second part of the statement.

  So let us assume that $X^*$ admits a weak$^*$ $\Delta$-point $x^*$.
  Looking back at the second part of the proof of
  Theorem~\ref{thm:D-point/weak*-Delta-points_implies_weak*-super-Delta-point},
  it is clear that we can use weak$^*$ sequentially compactness to
  extract sequences instead of nets at each key step, so that
  we end up with a sequential weak$^*$ super $\Delta$-point $y^*$.
  The conclusion follows.
\end{proof}

Stegall proved in \cite[Theorem~3.5]{Stegall81} that if a Banach space $X$
is weak Asplund, then the unit ball of $X^*$ is weak$^*$ sequentially compact.
So every Asplund space satisfies this property, and the same goes for
weakly compactly generated spaces by the results from Asplund's
seminal paper \cite{Asplund68}.

Smith and Sullivan proved in \cite{MR0458125} that
Hahn--Banach smooth spaces are Asplund,
so for Hahn--Banach smooth spaces, we get the following corollary.

\begin{cor}\label{cor:HBS_D-point_weak*-Delta_dual-seq-super-Delta}
  Let $X$ be a Hahn--Banach smooth space.
  If $X$ admits a $\mathfrak{D}$-point or if $X^*$ admits a weak$^*$
  $\Delta$-point, then $X^*$ contains a sequential super
  $\Delta$-point. In particular, $X^*$ fails to be Kadets--Klee.
\end{cor}

In particular, note that in the reflexive setting,
we even get the following result.

\begin{cor}\label{cor:reflexive_Kadets-Klee}
  Let $X$ be a reflexive Banach space.
  If $X$ contains a $\mathfrak{D}$-point,
  then $X$ and $X^*$ both contain a sequential super $\Delta$-point.
  In particular, neither $X$ nor $X^*$ are Kadets--Klee.
\end{cor}

\begin{proof}
  As reflexive spaces are Hahn--Banach smooth, the existence of a
  $\mathfrak{D}$-point in $X$ implies the existence of a sequential
  super $\Delta$-point in $X^*$ by
  Corollary~\ref{cor:HBS_D-point_weak*-Delta_dual-seq-super-Delta}.
  But as super $\Delta$-points are also $\mathfrak{D}$-points,
  this in turns implies the existence of a sequential super
  $\Delta$-point in $X$.
\end{proof}

\begin{rem}
  In particular, note that this result shows that it is no coincidence
  that the example from Section~\ref{sec:superr-banach-space},
  which is the first example of a reflexive Banach space with
  a $\Delta$-point, in fact provides an example of a reflexive space
  with a super $\Delta$-point and a super $\Delta$-point in its dual.

  Also let us point out,
  in relation to the discussions from
  Subsection~\ref{subsec:Delta-duality_and_asymptotic-geometry},
  that it was proved by Montesinos \cite[Theorem 3]{Montesinos87} that
  a reflexive Banach space $X$ has property $(\alpha)$ if and only if
  it is Kadets--Klee
  (the latter being referred to there as
  ``property $(H)$ of Radon-Riesz").
  So Corollary~\ref{cor:reflexive_Kadets-Klee} can be seen as
  an improved version of Corollary~\ref{cor:auc_refl_no_Dpts}.
\end{rem}

Finally, we get some general results for Lipschitz-free spaces with
a Hahn--Banach smooth predual. For Lipschitz-free spaces the \RNP and the Schur property are
equivalent by \cite[Theorem~4.6]{AGPP}. Hence, as noted above,
since Hahn-Banach smooth spaces are Asplund, we get from
Corollary~\ref{cor:HBS_D-point_weak*-Delta_dual-seq-super-Delta} the following.

\begin{cor}\label{cor:free-spaces_HBS-preduals}
  Let $M$ be a metric space such that $\lipfree{M}$
  is a dual space with a Hahn--Banach smooth predual $Y$.
  Then $Y$ does not contain any $\mathfrak{D}$-point and
  $\mathcal{F}(M)$ does not contain any weak$^*$ $\Delta$-point.
\end{cor}

Let $\mathcal{M}$ be the metric space defined by Veeorg
that we studied in Section~\ref{sec:separable-dual-space}.
Since $\lipfree{\mathcal{M}}$ contains a Daugavet point, we get from
Corollary~\ref{cor:free-spaces_HBS-preduals}
that no predual of $\lipfree{\mathcal{M}}$ is Hahn--Banach smooth.
Similarly one can show that for the metric space $M$
in \cite[Example~4.2]{ANPP} the molecule $m_{0q}$
is a $\Delta$-point (but not a Daugavet point),
hence no predual of $\lipfree{M}$ is Hahn--Banach smooth.
In next section, we will study in more detail $\Delta$-points
in Lipschitz-free spaces. In particular, we will prove in Theorem~\ref{thm:delta_molecules} 
that for any two distinct points $x$ and $y$ in a metric space $M$, the molecule $m_{x,y}$
is a $\Delta$-point if and only if $x$ and $y$ are discretely connectable in $M$
(see Definition~\ref{defn:discretely_connectable}).
So let us point out right away that together with
this result, Corollary~\ref{cor:free-spaces_HBS-preduals} also provides new information
about the metric spaces $M$ such that the associated Lipschitz-free space
has a Hahn--Banach smooth predual.

\begin{cor}\label{cor:disconnected_metric_HBS_preduals}

  Let $M$ be a metric space such that $\lipfree{M}$
  is a dual space with a Hahn--Banach smooth predual $Y$.
  Then no two distinct points in $M$ are discretely connectable.
    
\end{cor}

We started this subsection with an example showing that
non-reflexive M-embedded Banach spaces and their duals
can contain $\Delta$-points.
We end with the corresponding question for Daugavet-points.

\begin{quest}
  Does there exist a non-reflexive M-embedded space $X$
  such that $X$ or $X^*$ contains a Daugavet-point?
\end{quest}

\section{Metric characterization of
  \texorpdfstring{$\Delta$}{Delta}-molecules}
\label{sec:metr-char-delta-molecules}

In this last section we will focus on $\Delta$-points in
Lipschitz-free spaces.
Our main goal is to obtain a purely metric characterization of those
molecules $m_{xy}$ that are $\Delta$-points of $\lipfree{M}$ (see
Theorem~\ref{thm:delta_molecules}).
In \cite[Proposition~4.2]{JungRueda}, Jung and Rueda Zoca gave a
sufficient metric condition: that $x$ and $y$ be \emph{connectable},
that is, that they can be joined by Lipschitz paths in $M$ whose
length is arbitrarily close to $d(x,y)$.
In particular, if $x$ and $y$ are connected by a geodesic then
$m_{xy}$ is a $\Delta$-point.
Moreover, the converse is true when $M$ is compact under some
additional assumptions \cite[Theorem~4.13]{JungRueda}.
We will show below that these extra hypotheses are, in fact,
superfluous and the existence of a geodesic characterizes
$\Delta$-molecules for proper $M$
(see Corollary~\ref{cor:cmsp-no-delta-mpq}).
For general $M$, the notion of connectability is unnecessarily strong
and it may be relaxed to allow for discrete paths as follows.

\begin{defn}\label{defn:discretely_connectable}
  Let $x,y\in M$.
  Given $\varepsilon>0$, we say that $x$ and $y$ are
  \emph{$\varepsilon$-discretely connectable (in $M$)}
  if there exists a finite sequence of points
  $p_0,p_1,\ldots,p_n,p_{n+1}$ in $M$,
  where $p_0=x$ and $p_{n+1}=y$, with the following properties:
  \begin{enumerate}
  \item
    $d(p_i, p_{i+1}) < \varepsilon$ for each $i = 0,\ldots,n$, and
  \item
    $\displaystyle \sum_{i=0}^n d(p_i,p_{i+1}) < d(x,y) + \varepsilon$.
  \end{enumerate}
  We say that $x$ and $y$ are \emph{discretely connectable}
  if they are $\varepsilon$-discretely connectable for every $\varepsilon>0$.
\end{defn}

In the case where $M$ is proper, this notion
is still equivalent to the existence of geodesics.

\begin{prop}\label{prop:compact_connectability}
  If $M$ is a proper metric space and $x\neq y\in M$,
  then $x$ and $y$ are discretely connectable if and only if
  they are connected by a geodesic.
\end{prop}

\begin{proof}
  Denote $I = [0,d(x,y)] \subset \R$
  and let $\mathcal{U}$ be a free ultrafilter on $\mathbb{N}$.

  For every $n \in \mathbb{N}$ we can find
  $p_0^n, p_1^n, \ldots, p^n_{m(n)+1}$ in $M$ with
  $p_0^n = x$ and $p_{m(n)+1}^n = y$ with
  $d(p^n_i,p^n_{i+1}) < 1/n$ for $0\leq i\leq m(n)$ and
  \begin{equation*}
    \sum_{i=0}^{m(n)} d(p^n_i, p^n_{i+1}) < d(x,y) + \frac{1}{n}.
  \end{equation*}
  Fix $n \in \mathbb{N}$.
  For every $a \in I$ there exists $k = k_n(a)$ with $1 \le k \le m(n)+1$
  such that
  \begin{equation*}
    a < \sum_{i=0}^{k-1} d(p^n_i, p^n_{i+1}) \le a + \frac{1}{n}.
  \end{equation*}
  Define $p^n_a = p^n_k$. Then
  \begin{equation*}
    d(x,p^n_a) \le
    \sum_{i=0}^{k-1} d(p^n_i,p^n_{i+1}) \le a + \frac{1}{n}
  \end{equation*}
  and
  \begin{equation*}
    d(p^n_a,y) \le
    \sum_{i=k}^{m(n)} d(p^n_i,p^n_{i+1})
    =
    \sum_{i=0}^{m(n)}  d(p^n_i,p^n_{i+1})
    -
    \sum_{i=0}^{k-1}  d(p^n_i,p^n_{i+1})
    < d(x,y) + \frac{1}{n} - a,
  \end{equation*}
  so that
  \begin{equation*}
    p^n_a
    \in
    B\left(x, a + \frac{1}{n}\right)
    \cap
    B\left(y, d(x,y) - a +\frac{1}{n}\right).
  \end{equation*}
  Note that all points under consideration belong
  to the compact set $B(x,d(x,y)+1)$, so we can
  define $p_a = \lim_{\mathcal{U}} p^n_a$.
  In particular, we must have
  $x = p_0$ and $y = p_{d(x,y)}$.

  Let us check that the map $a\mapsto p_a$ from $I$ to $M$
  is an isometry.
  Let $a, b \in I$ with $a < b$. Given $\delta > 0$ choose
  $A \in \mathcal{U}$ such that $p^n_a \in B(p_a,\delta)$
  for all $n \in A$ and
  $B \in \mathcal{U}$ such that $p^n_b \in B(p_b,\delta)$
  for all $n \in B$.
  Let $C = A \cap B \in \mathcal{U}$.
  Since $\mathcal{U}$ is free we know that $C$ is infinite.
  Therefore we can choose $N \in C$ such that $1/N < \delta$.
  If $k_N(a) = k_N(b)$ then
  $p_a^N=p_b^N\in B(p_a,\delta)\cap B(p_b,\delta)$
  and $d(p_a,p_b)\leq 2\delta$.
  Otherwise $k_N(a) < k_N(b)$ and we have
  \begin{align*}
    d(p_a,p_b)
    &\le
    \delta + d(p_a^N,p_b^N) + \delta
    \le
    2\delta + \sum_{i=k_N(a)}^{k_N(b)-1} d(p_i^N,p_{i+1}^N) \\
    &=
    2\delta
    +
    \sum_{i=0}^{k_N(b)-1} d(p_i^N,p_{i+1}^N)
    -
    \sum_{i=0}^{k_N(a)-1} d(p_i^N,p_{i+1}^N)
    \\
    &<
    2\delta + b + \frac{1}N - a
    <
    b - a + 3\delta.
  \end{align*}
  Since $\delta > 0$ was arbitrary we get
  $d(p_a,p_b) \le |b-a|$ for all $a,b \in I$.
  Now, we also have
  \begin{equation*}
    d(x,y) \le d(x,p_a) + d(p_a,p_b) + d(p_b,y)
    \le (a-0) + (b-a) + (d(x,y) - b)
    = d(x,y)
  \end{equation*}
  so all inequalities must be equalities and
  $d(p_a,p_b) = |b-a|$ for all $a,b \in I$.
\end{proof}

The argument used in \cite[Proposition~4.2]{JungRueda} can be discretized to show that $m_{xy}$ is a $\Delta$-point whenever $x$ and $y$ are discretely connectable. With a slight variation of that argument, we get the following, more general result.

\begin{prop}\label{prop:discr_conn}
Let $\mu\in S_{\lipfree{M}}$. Suppose that, for every $\eta>0$, $\mu$ can be expressed as a series of molecules
\begin{equation}
\label{eq:series_of_molecules}
\mu = \sum_{k=1}^\infty a_k m_{x_ky_k} \quad \text{with} \quad \sum_{k=1}^\infty\abs{a_k}<1+\eta
\end{equation}
such that each pair $(x_k,y_k)$ is discretely connectable in $M$. Then $\mu$ is a $\Delta$-point in $\lipfree{M}$.
\end{prop}

The hypothesis clearly holds if $\mu=m_{xy}$ where $x$ and $y$ are discretely connectable. More generally, it also holds if every pair of points in $\supp(\mu)\cup\set{0}$ is discretely connectable in $M$, as every $\mu\in S_{\lipfree{M}}$ admits an expression of the form \eqref{eq:series_of_molecules} (see e.g. \cite[Lemma~2.1]{AP20}).

\begin{proof}[Proof of Proposition~\ref{prop:discr_conn}]
  Let $S = \set{\nu \in B_{\lipfree{M}} : f(\nu) > 1-\alpha}$
  be a slice containing $\mu$,
  for some $f\in S_{\Lip_0(M)}$ and $\alpha > 0$.
  Fix $\eta > 0$ such that
  $f(\mu) > (1 - \alpha)(1 + \eta)$,
  and choose a representation of $\mu$ of the form \eqref{eq:series_of_molecules} where every pair $(x_k,y_k)$ is discretely connectable. We may assume that $a_k\geq 0$ for all $k$ by swapping $x_k$ with $y_k$ if needed.
  Then, by convexity, we must have $f(m_{x_ky_k}) > 1 - \alpha$
  for some $k$, i.e. there are $x = x_k,y = y_k$ in $M$ that are
  discretely connectable and such that $m_{xy}\in S$.
  
  Now fix $\delta>0$ such that $f(m_{xy})>(1-\alpha)(1+\delta)$,
  and let $\varepsilon < \delta \cdot d(x,y)$ be arbitrary.
  Choose a sequence of points $p_1,\ldots,p_n\in M$
  as in Definition~\ref{defn:discretely_connectable},
  and denote $p_0 = x$, $p_{n+1} = y$.
  Then we have
  \begin{align*}
    \max \set{f(m_{p_i,p_{i+1}}) : i=0,\ldots,n}
    &=
    \max\set{\frac{f(p_i) - f(p_{i+1})}{d(p_i,p_{i+1})} :
      i=0,\ldots,n} \\
    &\geq
    \frac{\sum\limits_{i=0}^n (f(p_i) - f(p_{i+1}))}
    {\sum\limits_{i=0}^n d(p_i,p_{i+1})} \\
    &>
    \frac{f(x)-f(y)}{d(x,y)+\varepsilon} \\
    &>
    \frac{f(m_{xy})}{1+\delta} > 1-\alpha .
  \end{align*}
  Therefore we may choose $u=p_k,v=p_{k+1}$ with $d(u,v)<\varepsilon$
  and $f(m_{uv}) > 1 -\alpha$, i.e. $m_{uv}\in S$.
  By \cite[Theorem~2.6]{JungRueda},
  it follows that $S$ contains elements whose distance to $\mu$
  is arbitrarily close to $2$, hence $\mu$ is a $\Delta$-point.
\end{proof}

Discrete connectability does, in fact, characterize $\Delta$-molecules in Lipschitz-free spaces. In order to prove this, we will now construct a family of alternative metrics on $M$ that provide
information about how ``well connected'' (in the sense of Definition
\ref{defn:discretely_connectable}) a given pair of points is, by reducing their distance whenever there is a partial discrete path between them. We define them precisely as the shortest possible distance when giving a preference to discrete paths with sufficiently small step.

Fix $\alpha \in (0,1)$.
For any $\varepsilon > 0$ and $x,y \in M$ we write
\begin{equation*}
  w_{\alpha,\varepsilon}(x,y)
  :=
  \begin{cases}
    d(x,y) & \text{, if $d(x,y)\geq\varepsilon$} \\
    (1-\alpha)d(x,y) & \text{, if $d(x,y)<\varepsilon$}
  \end{cases}
\end{equation*}
and
\begin{equation*}
  b_{\alpha,\varepsilon}(x,y)
  :=
  \inf\set{\sum_{i=0}^n w_{\alpha,\varepsilon}(p_i,p_{i+1}) \,:
    \, p_0,p_1,\ldots,p_{n+1}\in M, p_0=x, p_{n+1}=y} .
\end{equation*}
Note that $w_{\alpha,\varepsilon}$ and $b_{\alpha,\varepsilon}$
increase as $\varepsilon$ decreases, so we can also define
  \begin{equation*}
    b_\alpha(x,y)
    :=
    \sup_{\varepsilon>0} b_{\alpha,\varepsilon}(x,y)
    =
    \lim_{\varepsilon\to 0} b_{\alpha,\varepsilon}(x,y) .
  \end{equation*}

\begin{lem}\label{lem:properties_of_b}
  Fix $\alpha\in (0,1)$ and $\varepsilon>0$.
  % \hfill
  \begin{enumerate}
  \item\label{item:prop_of_b_1}
    $b_\alpha$ and $b_{\alpha,\varepsilon}$ are bi-Lipschitz equivalent metrics on $M$;
  \item\label{item:prop_of_b_2}
    For any $x,y\in M$ we have
    \begin{equation*}
      (1-\alpha) d(x,y) \leq b_{\alpha,\varepsilon}(x,y)
      \leq b_\alpha(x,y) \leq d(x,y);
    \end{equation*}
  \item\label{item:prop_of_b_3}
    If $x,y\in M$ are not $\varepsilon$-discretely connectable,
    then
    \begin{equation*}
      b_{\alpha,\varepsilon}(x,y) \geq (1-\alpha)d(x,y) + \varepsilon\cdot\min\set{\alpha,1-\alpha} ;
    \end{equation*}
  \item\label{item:prop_of_b_4}
    Two points $x,y\in M$ are discretely connectable if and only if
    \begin{equation*}
      b_\alpha(x,y) = (1-\alpha)d(x,y).
    \end{equation*}
  \end{enumerate}
\end{lem}

\begin{proof}
  For any $x,y\in M$ we have
  $b_{\alpha,\varepsilon}(x,y) \leq w_{\alpha,\varepsilon}(x,y) \leq d(x,y)$.
  Notice also that for any finite sequence
  $p_0 = x, p_1, \ldots, p_n, p_{n+1} = y$ in $M$
  we have
  \begin{equation*}
    \sum_{i=0}^n w_{\alpha,\varepsilon}(p_i,p_{i+1})
    \geq
    \sum_{i=0}^n (1-\alpha)d(p_i,p_{i+1})
    \geq
    (1 - \alpha)d(x,y) .
  \end{equation*}
  and hence $b_{\alpha,\varepsilon}(x,y)\geq (1-\alpha)d(x,y)$.
  This proves \ref{item:prop_of_b_2} for $b_{\alpha,\varepsilon}$
  and thus also for $b_\alpha$ by taking limits.

  It is clear that $w_{\alpha,\varepsilon}$ and
  $b_{\alpha,\varepsilon}$ are symmetric.
  For any $x,y,z \in M$ and $\delta > 0$,
  we may find two finite sequences
  $p_1, \ldots, p_n$ and $p'_1, \ldots, p'_m$
  of points in $M$ such that
  \begin{align*}
    w_{\alpha,\varepsilon}(x,p_1) + w_{\alpha,\varepsilon}(p_1,p_2)
    + \cdots + w_{\alpha,\varepsilon}(p_n,z)
    &<
    b_{\alpha,\varepsilon}(x,z)+\delta \\
    w_{\alpha,\varepsilon}(z,p'_1) + w_{\alpha,\varepsilon}(p'_1,p'_2)
    + \cdots + w_{\alpha,\varepsilon}(p'_m,y)
    &<
    b_{\alpha,\varepsilon}(z,y)+\delta
  \end{align*}
  and therefore
  \begin{align*}
    b_{\alpha,\varepsilon}(x,y)
    &\leq
    w_{\alpha,\varepsilon}(x,p_1) + \cdots + w_{\alpha,\varepsilon}(p_n,z)
    +
    w_{\alpha,\varepsilon}(z,p'_1) + \cdots + w_{\alpha,\varepsilon}(p'_m,y) \\
    &<
    b_{\alpha,\varepsilon}(x,z) + b_{\alpha,\varepsilon}(z,y) + 2\delta .
  \end{align*}
  Letting $\delta \to 0$ yields the triangle inequality
  for $b_{\alpha,\varepsilon}$.
  Together with \ref{item:prop_of_b_2} this shows that
  $b_{\alpha,\varepsilon}$ is an equivalent metric on $M$,
  and letting $\varepsilon \to 0$ we get \ref{item:prop_of_b_1}.

  For part \ref{item:prop_of_b_3},
  the assumption is that for any finite sequence
  $p_0 = x, p_1, \ldots, p_{n+1} = y$ in $M$
  at least one of the two following statements holds:
  \begin{enumerate}[label={(\alph*)}]
  \item\label{item:prop_of_b_3_1}
    $\sum_{i=0}^n d(p_i,p_{i+1})
    \geq
    d(x,y)+\varepsilon$.
  \item\label{item:prop_of_b_3_2}
    $d(p_k,p_{k+1}) \geq \varepsilon$
    for some $k \in \set{0, \ldots, n}$.
  \end{enumerate}
  In case \ref{item:prop_of_b_3_1}, we have
  \begin{equation*}
    \sum_{i=0}^n w_{\alpha,\varepsilon}(p_i,p_{i+1})
    \geq
    \sum_{i=0}^n (1 - \alpha) d(p_i,p_{i+1})
    \geq
    (1-\alpha) (d(x,y)+\varepsilon)
    \geq
    (1-\alpha) d(x,y) + (1-\alpha)\varepsilon .
  \end{equation*}
  In case \ref{item:prop_of_b_3_2}, we have
  \begin{align*}
    \sum_{i=0}^n w_{\alpha,\varepsilon}(p_i,p_{i+1})
    &\geq
    (1 - \alpha)
    \sum_{i=0}^{k-1} d(p_i,p_{i+1})
    +
    d(p_k,p_{k+1})
    +
    (1-\alpha) \sum_{i=k+1}^n d(p_i,p_{i+1}) \\
    &=
    (1-\alpha)
    \sum_{i=0}^n d(p_i,p_{i+1})
    +
    \alpha d(p_k,p_{k+1}) \\
    &\geq
    (1-\alpha) d(x,y) + \alpha\varepsilon .
  \end{align*}
  Taking the infimum over all choices of $p_i$ yields
  \ref{item:prop_of_b_3}.

  Finally, one of the implications in \ref{item:prop_of_b_4}
  is given by \ref{item:prop_of_b_3}.
  For the converse,
  assume that $x$ and $y$ are discretely connectable.
  Let $\varepsilon' > 0$ and $\delta\in(0,\varepsilon')$.
  Then one may find finitely many points
  $p_0 = x, p_1, \ldots, p_n, p_{n+1} = y$ in $M$
  such that $d(p_i, p_{i+1}) < \delta$ and
  $\sum_{i=0}^n d(p_i,p_{i+1}) < d(x,y) + \delta$.
  Thus
  \begin{equation*}
    b_{\alpha,\varepsilon'} (x,y)
    \leq
    \sum_{i=0}^n w_{\alpha,\varepsilon'} (p_i, p_{i+1})
    =
    (1-\alpha) \sum_{i=0}^n d(p_i,p_{i+1})
    <
    (1-\alpha)(d(x,y) + \delta) .
  \end{equation*}
  Letting $\delta \to 0$ followed by $\varepsilon' \to 0$
  yields $b_\alpha(x,y)\leq (1-\alpha)d(x,y)$, and
  an appeal to \ref{item:prop_of_b_2} ends the proof.
\end{proof}

Note that Lemma~\ref{lem:properties_of_b}~\ref{item:prop_of_b_2}
implies that $\lipfree{M,d}$ and $\lipfree{M,b_\alpha}$
are linearly isomorphic completions of $\mathrm{span}\,\delta(M)$,
and in particular $\|\cdot\|_{\lipfree{M,b_\alpha}}$
is an equivalent norm on $\lipfree{M}$.
The same holds for $b_{\alpha,\varepsilon}$ in place of $b_\alpha$.
We have
\begin{equation*}
  (1-\alpha)\|\mu\|_{\mathcal{F}(M)}
  \le
  \|\mu\|_{\mathcal{F}(M,b_{\alpha,\varepsilon})}
  \le
  \|\mu\|_{\mathcal{F}(M,b_\alpha)}
  \le
  \|\mu\|_{\mathcal{F}(M)}
\end{equation*}
for $\mu\in\lipfree{M}$, and similarly
\begin{equation*}
  (1-\alpha)B_{\Lip_0(M,d)}
  \subseteq
  B_{\Lip_0(M,b_{\alpha,\varepsilon})}
  \subseteq
  B_{\Lip_0(M,b_\alpha)}
  \subseteq
  B_{\Lip_0(M,d)} .
\end{equation*}
We also have the following.

\begin{lem}\label{lem:norm_b_limit}
  For any $\mu\in\lipfree{M}$ and $\alpha\in (0,1)$ we have
  \begin{equation*}
  \lim_{\varepsilon\to 0}\norm{\mu}_{\lipfree{M,b_{\alpha,\varepsilon}}} = \norm{\mu}_{\lipfree{M,b_\alpha}} .
  \end{equation*}
\end{lem}

\begin{proof}
Suppose first that $\mu$ has finite support and put $S=\supp(\mu)\cup\set{0}$. Let $\eta>0$. Then, since $S$ is finite, we can find $\varepsilon_0>0$ such that $b_\alpha(x,y)\leq (1+\eta)b_{\alpha,\varepsilon}(x,y)$ for all $x,y\in S$ and all $\varepsilon\in (0,\varepsilon_0)$. For any such $\varepsilon$ we can, by e.g. \cite[Proposition~3.16]{Weaver2}, write $\mu$ as a finite sum of molecules in $\lipfree{M,b_{\alpha,\varepsilon}}$ in the form
\begin{equation*}
\mu = \sum_{k=1}^n a_k\frac{\delta(x_k)-\delta(y_k)}{b_{\alpha,\varepsilon}(x_k,y_k)}
\end{equation*}
where $x_k\neq y_k\in S$ and $\sum_k\abs{a_k}=\norm{\mu}_{\lipfree{M,b_{\alpha,\varepsilon}}}$. This implies
\begin{align*}
\norm{\mu}_{\lipfree{M,b_{\alpha,\varepsilon}}} \leq \norm{\mu}_{\lipfree{M,b_\alpha}} &= \norm{\sum_{k=1}^n a_k \frac{b_\alpha(x_k,y_k)}{b_{\alpha,\varepsilon}(x_k,y_k)} \frac{\delta(x_k)-\delta(y_k)}{b_\alpha(x_k,y_k)}}_{\lipfree{M,b_\alpha}} \\
&\leq \sum_{k=1}^n \abs{a_k} \frac{b_\alpha(x_k,y_k)}{b_{\alpha,\varepsilon}(x_k,y_k)} \leq (1+\eta)\norm{\mu}_{\lipfree{M,b_{\alpha,\varepsilon}}}
\end{align*}
for $\varepsilon<\varepsilon_0$. Thus the lemma holds for this $\mu$.

Now let $\mu\in\lipfree{M}$ be arbitrary, and take $\eta>0$. Find $\nu\in\lipfree{M}$ with finite support and such that $\norm{\mu-\nu}_{\lipfree{M,b_\alpha}}\leq\eta$. Then $\norm{\nu}_{\lipfree{M,b_\alpha}}\leq\norm{\nu}_{\lipfree{M,b_{\alpha,\varepsilon}}}+\eta$ when $\varepsilon$ is small enough, by the previous paragraph. For such $\varepsilon$ we have
\begin{align*}
\norm{\mu}_{\lipfree{M,b_{\alpha,\varepsilon}}} \leq \norm{\mu}_{\lipfree{M,b_\alpha}} &\leq \norm{\nu}_{\lipfree{M,b_\alpha}} + \eta \\
&\leq \norm{\nu}_{\lipfree{M,b_{\alpha,\varepsilon}}}+2\eta \\
&\leq \norm{\mu}_{\lipfree{M,b_{\alpha,\varepsilon}}} + \norm{\mu-\nu}_{\lipfree{M,b_{\alpha,\varepsilon}}} + 2\eta \\
&\leq \norm{\mu}_{\lipfree{M,b_{\alpha,\varepsilon}}} + \norm{\mu-\nu}_{\lipfree{M,b_\alpha}} + 2\eta \\
&\leq \norm{\mu}_{\lipfree{M,b_{\alpha,\varepsilon}}} + 3\eta .
\end{align*}
So the lemma also holds for this $\mu$.
\end{proof}

The relevance of the next lemma lies in the fact that condition~\ref{item:norm_b_1} characterizes $\Delta$-points when $\mu$ is a molecule (see \cite[Theorem~4.7]{JungRueda}) or, more generally, a finitely supported element of $S_{\lipfree{M}}$ (see \cite[Theorem~4.4]{VeeorgStudia}).

\begin{lem}\label{lem:norm_b}
  Let $\mu\in S_{\lipfree{M}}$.
  Then the following are equivalent:
  \begin{enumerate}
  \item\label{item:norm_b_1}
    Every slice of $B_{\lipfree{M}}$ that contains $\mu$ also
    contains molecules $m_{uv}$ for arbitrarily small $d(u,v)$.
  \item\label{item:norm_b_2}
    $\norm{\mu}_{\lipfree{M,b_{\alpha,\varepsilon}}} = 1-\alpha$
    for all $\alpha \in (0,1)$ and all $\varepsilon>0$.
  \item\label{item:norm_b_3}
    $\norm{\mu}_{\lipfree{M,b_\alpha}} = 1-\alpha$
    for all $\alpha \in (0,1)$.
  \end{enumerate}
\end{lem}

\begin{proof}
  \ref{item:norm_b_1} $\Rightarrow$ \ref{item:norm_b_2}.
  Assume that \ref{item:norm_b_1} holds and suppose
  $\|\mu\|_{\lipfree{M,b_{\alpha,\varepsilon}}} > 1 - \alpha$
  for some $\varepsilon > 0$ and $\alpha \in (0,1)$.
  Then there exists
  $h \in S_{\Lip_0(M,b_{\alpha,\varepsilon})} \subset B_{\Lip_0(M)}$
  such that $h(\mu) > 1 - \alpha$.
  By assumption there exists a molecule $m_{uv}$
  in $\mathcal{F}(M)$ such that $h(m_{uv}) > 1 - \alpha$
  and $d(u,v) < \varepsilon$.
  Thus
  \begin{equation*}
    b_{\alpha,\varepsilon}(u,v) \ge h(u)-h(v)> (1-\alpha)d(u,v)
    = w_{\alpha,\varepsilon}(u,v) \ge b_{\alpha,\varepsilon}(u,v) .
  \end{equation*}
  This contradiction proves \ref{item:norm_b_2}.

  \ref{item:norm_b_2} $\Rightarrow$ \ref{item:norm_b_3}
  follows from Lemma~\ref{lem:norm_b_limit}.

  \ref{item:norm_b_3} $\Rightarrow$ \ref{item:norm_b_1}.
  Assume that \ref{item:norm_b_3} holds and
  fix $\varepsilon > 0$
  and a slice $S(f, \alpha)$
  such that $\mu\in S(f,\alpha)$,
  where $f \in S_{\Lip_0(M)}$ and $\alpha > 0$.
  By \cite[Lemma 2.1]{zbMATH02168839}, we may assume $\alpha\in (0,1)$.
  From \ref{item:norm_b_3} we have
  \begin{equation*}
    \|f\|_{\Lip_0(M,b_\alpha)}
    \ge
    \frac{f(\mu)}{\|\mu\|_{\lipfree{M,b_\alpha}}}
    >
    \frac{1 - \alpha}{1 - \alpha}
    =
    1.
  \end{equation*}
  Thus there exist $x,y\in M$ such that
  \begin{equation*}
    f(x)-f(y) > b_\alpha(x,y) \geq b_{\alpha,\varepsilon}(x,y).
  \end{equation*}
  By the definition of $b_{\alpha,\varepsilon}(x,y)$ we can find
  $p_0, p_1, \ldots, p_{n+1}\in M$
  such that $p_0 = x, p_{n+1} = y$ and
  \begin{equation*}
    f(x) - f(y)
    >
    \sum_{i=0}^n w_{\alpha,\varepsilon}(p_i,p_{i+1}).
  \end{equation*}
  Let
  $I_1 = \big\{ i \in\{0,\ldots,n\} :
  d(p_i,p_{i+1}) < \varepsilon \big\}$
  and let
  $I_2=\{0,\ldots,n\}\setminus I_1$.
  Then
  \begin{align*}
    (1-\alpha)
    \sum_{i\in I_1} d(p_i,p_{i+1}) + \sum_{i\in I_2} d(p_i,p_{i+1})
    &=
    \sum_{i=0}^n w_{\alpha,\varepsilon}(p_i,p_{i+1})\\
    &<
    f(x)-f(y)\\
    &=
    \sum_{i=0}^n \big( f(p_i) - f(p_{i+1}) \big)\\
    &\le
    \sum_{i\in I_1} \big( f(p_i) - f(p_{i+1} ) \big)
    +
    \sum_{i\in I_2} d(p_i,p_{i+1})
  \end{align*}
  and therefore there exists $i\in I_1$ such that
  $f(p_i) - f(p_{i+1}) > (1 - \alpha) d(p_i,p_{i+1})$.
  Since we have also $d(p_i,p_{i+1}) < \varepsilon$,
  we conclude that \ref{item:norm_b_1} holds with $u=p_i$, $v=p_{i+1}$.
\end{proof}

We are now in a position to prove our characterization of $\Delta$-molecules.

\begin{thm}\label{thm:delta_molecules}
  Let $x\neq y\in M$.
  Then $m_{xy}$ is a $\Delta$-point of $\lipfree{M}$
  if and only if
  $x$ and $y$ are discretely connectable in $M$.
\end{thm}

\begin{proof}
  One implication follows immediately from Proposition~\ref{prop:discr_conn}.
  For the converse, suppose that $m_{xy}$ is a $\Delta$-point
  and fix $\alpha\in (0,1)$.
  Then $\mu=m_{xy}$ satisfies condition~\ref{item:norm_b_1}
  from Lemma~\ref{lem:norm_b} by \cite[Theorem~4.7]{JungRueda}, so it
  also satisfies \ref{item:norm_b_3} and
  \begin{equation*}
    1-\alpha
    =
    \norm{m_{xy}}_{\lipfree{M,b_\alpha}}
    =
    \norm{\frac{\delta(x)-\delta(y)}{d(x,y)}}_{\lipfree{M,b_\alpha}}
    =
    \frac{b_\alpha(x,y)}{d(x,y)}.
  \end{equation*}
  That is, $b_\alpha(x,y)=(1-\alpha)d(x,y)$.
  Now Lemma~\ref{lem:properties_of_b}~\ref{item:prop_of_b_4}
  shows that $x$ and $y$ are discretely connectable.
\end{proof}

\begin{cor}\label{cor:cmsp-no-delta-mpq}
  Let $M$ be a proper metric space and $x\neq y\in M$.
  Then $m_{xy}$ is a $\Delta$-point of $\lipfree{M}$
  if and only if
  $x$ and $y$ are connected with a geodesic.
\end{cor}

By \cite[Corollary~4.5]{VeeorgStudia}, every convex sum (finite or infinite) of $\Delta$-molecules of $\lipfree{M}$ is again a $\Delta$-point. It is then natural to ask whether the converse holds.
Note that the question only makes sense for elements of $S_{\lipfree{M}}$ that are actually convex sums of molecules, which is not all of them in general (see \cite[Section 4]{APS}).
This question was raised explicitly in \cite[Problem~3]{VeeorgStudia} for those $\Delta$-points $\mu\in S_{\lipfree{M}}$ with finite support. In that case we have:
\begin{itemize}
\item $\mu$ can always be written as a convex sum of molecules (see e.g. \cite[Proposition~3.16]{Weaver2}), and
\item $\mu$ also satisfies  property~\ref{item:norm_b_1} from Lemma~\ref{lem:norm_b}, by \cite[Theorem~4.4]{VeeorgStudia}.
\end{itemize}
The techniques developed in this section allow us to answer the question in the positive.

\begin{thm}[cf. {\cite[Problem 3]{VeeorgStudia}}]
  \label{thm:delta_finite_conv_comb}
  Suppose that $\mu\in S_{\lipfree{M}}$
  is a $\Delta$-point with finite support.
  Then $\mu$ can be written as a finite convex combination
  of $\Delta$-molecules in $S_{\lipfree{M}}$.
\end{thm}

\begin{proof}
  Fix $\alpha\in (0,1)$.
  By Lemma~\ref{lem:norm_b} and \cite[Theorem~4.4]{VeeorgStudia}
  we have
  $\norm{\mu}_{\lipfree{M,b_\alpha}} = 1-\alpha$.
  Thus, since $\mu$ is a finitely supported element of
  $\lipfree{M,b_\alpha}$,
  we may write $\mu/\norm{\mu}_{\lipfree{M,b_\alpha}}$ as a finite convex combination
  of $b_\alpha$-molecules (e.g. by \cite[Proposition~3.16]{Weaver2}).
  That is,
  \begin{equation*}
    \frac{\mu}{1-\alpha}
    =
    \sum_{i=1}^n \lambda_i \frac{\delta(x_i)-\delta(y_i)}{b_\alpha(x_i,y_i)}
  \end{equation*}
  for some $x_i \neq y_i \in M$
  and $\lambda_i > 0$ such that
  $\sum_{i=1}^n \lambda_i = 1$.
  Then
  \begin{align*}
    1
    =
    \norm{\mu}_{\lipfree{M,d}}
    &=
    (1-\alpha) \norm{\sum_{i=1}^n \lambda_i
      \frac{d(x_i, y_i)}{b_\alpha(x_i,y_i)}
      \frac{\delta(x_i) - \delta(y_i)}{d(x_i,y_i)}
    }_{\lipfree{M,d}} \\
    &\leq
    \sum_{i=1}^n \lambda_i\cdot (1-\alpha) \frac{d(x_i,y_i)}{b_\alpha(x_i,y_i)}
    \leq
    \sum_{i=1}^n \lambda_i\cdot\frac{1-\alpha}{1-\alpha}
    =
    1
  \end{align*}
  and so all inequalities are actually equalities.
  In particular $\mu=\sum_{i=1}^n \lambda_i m_{x_iy_i}$
  is a finite convex combination of molecules such that
  $b_\alpha(x_i,y_i) = (1 - \alpha)d(x_i,y_i)$ for all $i$.
  By Lemma~\ref{lem:properties_of_b}~\ref{item:prop_of_b_4}
  and Theorem~\ref{thm:delta_molecules},
  this means that each $m_{x_iy_i}$ is a $\Delta$-point.
\end{proof}

We finish by remarking that the existence of $\Delta$-points in
$\lipfree{M}$ does not necessarily imply the existence of
$\Delta$-molecules in general.
For instance, if $M$ is the Smith--Volterra--Cantor set then
$\lipfree{M}$ is isometric to $L_1\oplus_1\ell_1$
by the proof of \cite[Corollary 3.4]{Godard},
which admits $\Delta$-points since $L_1$ does.
However, $M$ is compact and totally disconnected,
so $\lipfree{M}$ cannot contain $\Delta$-molecules.
In particular, the converse of Proposition \ref{prop:discr_conn} does not hold.

\section*{Acknowledgments}
\label{sec:ack}

This work was supported by: \begin{itemize}

  \item the AURORA mobility programme (project
  number: 309597) from the Norwegian Research Council and the ``PHC
  Aurora'' program (project number: 45391PF), funded by the French
  Ministry for Europe and Foreign Affairs, the French Ministry for
  Higher Education, Research and Innovation;

  \item the Grant BEST/2021/080 funded by the Generalitat Valenciana, Spain, 
  and by Grant PID2021-122126NB-C33 
  funded by MCIN/AEI/10.13039/501100011033 and by 
  ``ERDF A way of making Europe'';

  \item the Estonian Research Council grants PRG 1901 and SJD58;

  \item the French ANR project No. ANR-20-CE40-0006.
  
  \end{itemize}

Parts of this research were conducted while
T. A. Abrahamsen, R. J. Aliaga and V. Lima visited
the Laboratoire de Math\'ematiques de Besan\c{c}on in 2021 and
while R. J. Aliaga visited the
Institute of Mathematics and Statistics at the University of Tartu 
in 2022, for which they wish to express their gratitude.

\newcommand{\etalchar}[1]{$^{#1}$}
\providecommand{\bysame}{\leavevmode\hbox to3em{\hrulefill}\thinspace}
\providecommand{\MR}{\relax\ifhmode\unskip\space\fi MR }
\providecommand{\MRhref}[2]{%
  \href{http://www.ams.org/mathscinet-getitem?mr=#1}{#2}
}
\providecommand{\href}[2]{#2}


\begin{thebibliography}{DKR{\etalchar{+}}16}

\bibitem[AACD21]{AACD_sums}
F.~Albiac, J.~L. Ansorena, M.~C\'{u}th, and M.~Doucha, \emph{Lipschitz free
  spaces isomorphic to their infinite sums and geometric applications}, Trans.
  Amer. Math. Soc. \textbf{374} (2021), no.~10, 7281--7312. \MR{4315605}

\bibitem[AGPP22]{AGPP}
R.~J. Aliaga, C.~Gartland, C.~Petitjean, and A.~Proch\'{a}zka, \emph{Purely
  1-unrectifiable metric spaces and locally flat {L}ipschitz functions}, Trans.
  Amer. Math. Soc. \textbf{375} (2022), no.~5, 3529--3567. \MR{4402669}

\bibitem[AHLP20]{AHLP}
T.~A. Abrahamsen, R.~Haller, V.~Lima, and K.~Pirk, \emph{Delta- and {D}augavet
  points in {B}anach spaces}, Proc. Edinb. Math. Soc. (2) \textbf{63} (2020),
  no.~2, 475--496. \MR{4085036}

\bibitem[AK16]{AlbiacKalton}
F.~Albiac and N.~J. Kalton, \emph{Topics in {B}anach space theory}, second ed.,
  Graduate Texts in Mathematics, vol. 233, Springer, [Cham], 2016, With a
  foreword by Gilles Godefory. \MR{3526021}

\bibitem[ALL16]{MR3415738}
T.~A. Abrahamsen, J.~Langemets, and V.~Lima, \emph{Almost square {B}anach
  spaces}, J. Math. Anal. Appl. \textbf{434} (2016), no.~2, 1549--1565.
  \MR{3415738}

\bibitem[ALM22]{ALM}
T.~A. Abrahamsen, V.~Lima, and A.~Martiny, \emph{Delta-points in {B}anach
  spaces generated by adequate families}, Illinois J. Math. \textbf{66} (2022),
  no.~3, 421--434. \MR{4477423}

\bibitem[ALMP22]{ALMP}
T.~A. Abrahamsen, V.~Lima, A.~Martiny, and Y.~Perreau, \emph{Asymptotic
  geometry and delta-points}, Banach J. Math. Anal. \textbf{16} (2022), no.~4,
  Paper No. 57, 33. \MR{4468598}

\bibitem[ALMT21]{ALMT}
T.~A. Abrahamsen, V.~Lima, A.~Martiny, and S.~Troyanski, \emph{Daugavet- and
  delta-points in {B}anach spaces with unconditional bases}, Trans. Amer. Math.
  Soc. Ser. B \textbf{8} (2021), 379--398. \MR{4249632}

\bibitem[ANPP21]{ANPP}
R.~J. Aliaga, C.~No\^{u}s, C.~Petitjean, and A.~Proch\'{a}zka, \emph{Compact
  reduction in {L}ipschitz-free spaces}, Studia Math. \textbf{260} (2021),
  no.~3, 341--359. \MR{4296732}

\bibitem[AP20]{AP20}
R.~J. Aliaga and E.~Perneck\'{a}, \emph{Supports and extreme points in
  {L}ipschitz-free spaces}, Rev. Mat. Iberoam. \textbf{36} (2020), no.~7,
  2073--2089. \MR{4163992}

\bibitem[APPP20]{APPP}
R.~J. Aliaga, E.~Perneck\'{a}, C.~Petitjean, and A.~Proch\'{a}zka,
  \emph{Supports in {L}ipschitz-free spaces and applications to extremal
  structure}, J. Math. Anal. Appl. \textbf{489} (2020), no.~1, 124128, 14.
  \MR{4083124}

\bibitem[APS23]{APS}
R.~J. Aliaga, E.~Perneck\'a, and R.~J. Smith, \emph{Convex integrals of
  molecules in {L}ipschitz-free spaces}, arXiv:2302.13951, 2023.

\bibitem[Asp68]{Asplund68}
E.~Asplund, \emph{Fr\'{e}chet differentiability of convex functions}, Acta
  Math. \textbf{121} (1968), 31--47. \MR{231199}

\bibitem[BK14]{BowersKalton}
A.~Bowers and N.~J. Kalton, \emph{An introductory course in functional
  analysis}, Universitext, New York, NY: Springer, 2014 (English).

\bibitem[CK10]{MR2646086}
S.~R. Cowell and N.~J. Kalton, \emph{Asymptotic unconditionality}, Q. J. Math.
  \textbf{61} (2010), no.~2, 217--240. \MR{2646086}

\bibitem[Dal15]{MR3376824}
A.~Dalet, \emph{Free spaces over some proper metric spaces}, Mediterr. J. Math.
  \textbf{12} (2015), no.~3, 973--986. \MR{3376824}

\bibitem[DGZ93]{DGZ}
R.~Deville, G.~Godefroy, and V.~Zizler, \emph{Smoothness and renormings in
  {B}anach spaces}, Longman Scientific \& Technical, Harlow, 1993.
  \MR{94d:46012}

\bibitem[DKR{\etalchar{+}}16]{DKR+}
S.~J. Dilworth, D.~Kutzarova, N.~L. Randrianarivony, J.~P. Revalski, and N.~V.
  Zhivkov, \emph{Lenses and asymptotic midpoint uniform convexity}, J. Math.
  Anal. Appl. \textbf{436} (2016), no.~2, 810--821. \MR{3446981}

\bibitem[FG88]{FG}
M.~Fabi\'{a}n and G.~Godefroy, \emph{The dual of every {A}splund space admits a
  projectional resolution of the identity}, Studia Math. \textbf{91} (1988),
  no.~2, 141--151. \MR{985081}

\bibitem[FHH{\etalchar{+}}11]{MR2766381}
M.~Fabian, P.~Habala, P.~H{\'a}jek, V.~Montesinos, and V.~Zizler, \emph{Banach
  space theory}, CMS Books in Mathematics/Ouvrages de Math\'ematiques de la
  SMC, Springer, New York, 2011, The basis for linear and nonlinear analysis.
  \MR{2766381 (2012h:46001)}

\bibitem[Fon00]{MR1792984}
V.~P. Fonf, \emph{On the boundary of a polyhedral {B}anach space}, Extracta
  Math. \textbf{15} (2000), no.~1, 145--154. \MR{1792984}

\bibitem[GM92]{MR1188888}
J.~R. Giles and W.~B. Moors, \emph{Differentiability properties of {B}anach
  spaces where the boundary of the closed unit ball has denting point
  properties}, Miniconference on probability and analysis ({S}ydney, 1991),
  Proc. Centre Math. Appl. Austral. Nat. Univ., vol.~29, Austral. Nat. Univ.,
  Canberra, 1992, pp.~107--115. \MR{1188888}

\bibitem[God10]{Godard}
A.~Godard, \emph{Tree metrics and their {L}ipschitz-free spaces}, Proc. Amer.
  Math. Soc. \textbf{138} (2010), no.~12, 4311--4320. \MR{2680057}

\bibitem[HWW93]{HWW}
P.~Harmand, D.~Werner, and W.~Werner, \emph{{$M$}-ideals in {B}anach spaces and
  {B}anach algebras}, Lecture Notes in Mathematics, vol. 1547, Springer-Verlag,
  Berlin, 1993. \MR{1238713}

\bibitem[IK04]{zbMATH02168839}
Y.~{Ivakhno} and V.~{Kadets}, \emph{{Unconditional sums of spaces with bad
  projections.}}, {Visn. Khark. Univ., Ser. Mat. Prykl. Mat. Mekh.}
  \textbf{645} (2004), no.~54, 30--35 (English).

\bibitem[JRZ22]{JungRueda}
M.~Jung and A.~Rueda~Zoca, \emph{Daugavet points and {$\Delta $}-points in
  {L}ipschitz-free spaces}, Studia Math. \textbf{265} (2022), no.~1, 37--55.
  \MR{4420901}

\bibitem[KLT22]{KLT}
A.~Kamińska, H.~J. Lee, and H.~J. Tag, \emph{Daugavet and diameter two
  properties in {O}rlicz-{L}orentz spaces}, arXiv:2212.12149, 2022.

\bibitem[KSSW00]{MR1621757}
V.~M. Kadets, R~V. Shvidkoy, G.~G. Sirotkin, and D.~Werner, \emph{Banach spaces
  with the {D}augavet property}, Trans. Amer. Math. Soc. \textbf{352} (2000),
  no.~2, 855--873. \MR{1621757}

\bibitem[KW04]{KW_Schur+Daugavet}
V.~Kadets and D.~Werner, \emph{A {B}anach space with the {S}chur and the
  {D}augavet property}, Proc. Amer. Math. Soc. \textbf{132} (2004), no.~6,
  1765--1773. \MR{2051139}

\bibitem[Lan95]{Lancien95}
G.~Lancien, \emph{On uniformly convex and uniformly {K}adec-{K}lee renormings},
  Serdica Math. J. \textbf{21} (1995), no.~1, 1--18. \MR{1333523}

\bibitem[Mon87]{Montesinos87}
V.~Montesinos, \emph{Drop property equals reflexivity}, Studia Math.
  \textbf{87} (1987), no.~1, 93--100. \MR{924764}

\bibitem[MPRZ23]{MPRZ}
M.~Mart\'{\i}n, Y.~Perreau, and A.~Rueda~Zoca, \emph{Diametral notions for
  elements of the unit ball of a {B}anach space}, arXiv:2301.04433, 2023.

\bibitem[MRZ22]{MR4405563}
M.~Mart\'{\i}n and A.~Rueda~Zoca, \emph{Daugavet property in projective
  symmetric tensor products of {B}anach spaces}, Banach J. Math. Anal.
  \textbf{16} (2022), no.~2, Paper No. 35, 32. \MR{4405563}

\bibitem[Pet18]{CP}
C.~Petitjean, \emph{Some aspects of the geometry of {L}ipschitz free spaces},
  Ph.D. thesis, Université {B}ourgogne {F}ranche-{Co}mté, 2018.

\bibitem[PP74]{PetuninPlichko}
Ju.~\={I}. Petun\={\i}n and A.~N. Pl\={\i}\v{c}ko, \emph{Some properties of the
  set of functionals that attain a supremum on the unit sphere}, Ukrain. Mat.
  \v{Z}. \textbf{26} (1974), 102--106, 143. \MR{0336299}

\bibitem[Rol87]{MR928575}
S.~Rolewicz, \emph{On {$\Delta$}-uniform convexity and drop property}, Studia
  Math. \textbf{87} (1987), no.~2, 181--191. \MR{928575}

\bibitem[SS77]{MR0458125}
M.~A. Smith and F.~Sullivan, \emph{Extremely smooth {B}anach spaces}, Banach
  spaces of analytic functions ({P}roc. {P}elczynski {C}onf., {K}ent {S}tate
  {U}niv., {K}ent, {O}hio, 1976), 1977, pp.~125--137. Lecture Notes in Math.,
  Vol. 604. \MR{0458125}

\bibitem[Ste81]{Stegall81}
C.~Stegall, \emph{The {R}adon-{N}ikod\'{y}m property in conjugate {B}anach
  spaces. {II}}, Trans. Amer. Math. Soc. \textbf{264} (1981), no.~2, 507--519.
  \MR{603779}

\bibitem[Sul77]{Sullivan}
F.~Sullivan, \emph{Geometrical peoperties determined by the higher duals of a
  {B}anach space}, Illinois J. Math. \textbf{21} (1977), no.~2, 315--331.
  \MR{458124}

\bibitem[Vee22]{VeeorgFunc}
T.~Veeorg, \emph{Daugavet- and {D}elta-points in spaces of {L}ipschitz
  functions}, arXiv:2206.03475, 2022.

\bibitem[Vee23]{VeeorgStudia}
\bysame, \emph{Characterizations of {D}augavet points and {D}elta-points in
  {L}ipschitz-free spaces}, Studia Mathematica \textbf{268} (2023), 213--233.

\bibitem[Wea18]{Weaver2}
N.~Weaver, \emph{Lipschitz algebras}, World Scientific Publishing Co. Pte.
  Ltd., Hackensack, NJ, 2018, Second edition of [ MR1832645]. \MR{3792558}

\end{thebibliography}
\end{document}